\documentclass[english]{amsart}

\usepackage[all,cmtip]{xy}

\usepackage{amsmath,amssymb,amscd,amsfonts}

\usepackage{babel}
\usepackage{color}
\usepackage{amstext}
\usepackage{amsmath}
\usepackage{amsfonts}
\usepackage{latexsym}
\usepackage{ifthen}
\usepackage{xypic}
\usepackage[all,cmtip]{xy}
\xyoption{all}
\renewcommand{\P}{\mathbb{P}}
\newcommand{\R}{\mathbb{R}}
\newcommand{\CC}{\mathbb{C}}

\newcommand{\N}{\mathbb{N}}

\newcommand{\cE}{\mathcal{E}}

\newcommand{\cG}{\mathcal{G}}
\newcommand{\cI}{\mathcal{I}}

\newcommand{\cF}{\mathcal{F}}
\newcommand{\cK}{\mathcal{K}}
\renewcommand{\O}{\mathcal{O}}

\newcommand{\ep}{\varepsilon}
\renewcommand{\epsilon}{\varepsilon}

\newcommand{\ol}{\overline}

\renewcommand{\ge}{\geqslant}

\renewcommand{\leq}{\leqslant}
\renewcommand{\geq}{\geqslant}

\newcommand{\End}{\mathrm{End}}
\newcommand{\wh}{\widehat}
\newcommand{\wt}{\widetilde}

\newcommand{\ddc}{dd^c}

\newcommand{\Id}{\mathrm {Id}}
\newcommand{\codim}{\mathrm {codim}}

\newcommand{\Sym}{\mathrm{Sym}}
\newcommand{\Div}{\mathrm{Div}}

\newcommand{\pr}{\mathrm{pr}}

\newcommand{\dbar}{\bar \partial}

\newtheorem{thm}{Theorem}[section]
\newtheorem{lemme}[thm]{Lemma}
\newtheorem{remark}[thm]{Remark}
\newtheorem{proposition}[thm]{Proposition}
\newtheorem{question}[thm]{Question}

\newtheorem{defn}[thm]{Definition}
\newtheorem{cor}[thm]{Corollary}
\newtheorem{claim}[thm]{Claim}
\title{Subharmonicity of direct images and applications}
\author[]{Fr\'ed\'eric CAMPANA, Junyan CAO, Mihai P{A}UN}

\address{Universit\'e de Lorraine\\  
  Institut de Math\'ematiques \'Elie Cartan, \\
  B.P. 70239, 54506 Vandoeuvre-les-Nancy  Cedex, France}
\email{frederic.campana@univ-lorraine.fr}

\address{Université Côte d'Azur, Laboratoire de Mathématiques J.A. Dieudonné, Parc Valrose  06108 Nice cedex 02, France}
\email{junyan.cao@unice.fr}

\address{Universit\"at Bayreuth, Mathematisches Institut, Lehrstuhl Mathematik VIII, Universit\"atsstrasse 30, D-95447, Bayreuth, Germany}
\email{mihai.paun@uni-bayreuth.de}

\begin{document}

\maketitle

\section{Introduction and main results}

\noindent The positivity properties of direct images are at the heart of many recent
developments in complex geometry. In the present article our main goal is to explore further some of these topics. Actually the main results we obtain here concern the \emph{subharmonicity properties of direct images}. As a consequence we answer partially to a conjecture formulated by Pereira-Touzet, cf. \cite{PeTou} and the references therein. We start with the application as follows.
\smallskip

\noindent Let $X$ be a compact K\"ahler manifold whose canonical bundle
$K_X$ is pseudo-effective. Let $\cF\subset T_X$ be a holomorphic foliation
such that its first Chern class $c_1(\cF)$ is trivial. Then Pereira-Touzet show (cf. \cite{PeTou}) that $\cF$ is non-singular, and they conjecture that if moreover $c_2(\cF)\neq 0$  and
$\cF$ is stable with respect to some K\"ahler metric $\omega_X$ on $X$, then
the leaves of $\cF$ are algebraic.
\medskip

\noindent In this article we obtain the following particular case of their conjecture.
\begin{thm}\label{thm001}
Let $(X, \omega_X)$ be a smooth projective manifold. Let $\cF\subset T_X$ be a holomorphic foliation such that the following hold.
\begin{enumerate}

\item[(i)] The first Chern class of $\cF$ is zero, i.e. $c_1(\cF)= 0$ and $c_2(\cF)\neq 0$.
  \smallskip

\item[(ii)] The sheaf $\wh{\Sym}^k\cF^\star$ is $\omega_X$-stable for some $k\geq r$, where $r$ is the rank of $\cF$. 
\end{enumerate}  
Then the leaves of $\cF$ are algebraic. 
\end{thm}
\noindent We refer to \cite{HP19}, \cite{Dr18}, \cite{LPT} and the references therein for other particular cases of the
aforementioned conjecture and its relevance in the context of singular
Calabi-Yau manifolds.
\smallskip

\noindent Now the main new ingredient in the proof of Theorem \ref{thm001} is
Theorem \ref{thm1} below that
we next discuss. We recall that if $\cF$ is any torsion free, coherent sheaf on $X$
then there exists a modification $\pi: \wh X\to X$ such that the inverse image
$\pi^\star(\cF)$ modulo torsion is a vector bundle which we denote by $E$. Let
$\O(1)\to \P(E)$ be the (dual of the) tautological bundle over $\P(E)$
(our convention here is that a point of $\P(E)$ corresponds to a line
of the fiber of $E$ at some point of $X$).
\smallskip

\noindent Then we have the following statement.

\begin{thm}\label{thm1}
Let $(X, \omega_X)$ be a compact K\"ahler manifold, and let $\cF$ be a reflexive sheaf such that the following hold.
\begin{enumerate}

\item[(i)] The first Chern class of $\cF$ is zero, i.e. $c_1(\cF)= 0$.
  \smallskip

\item[(ii)] The sheaf $\widehat{\Sym}^k\cF^\star$ is $\omega_X$-stable for some $k\geq r +1$, where $r$ is the rank of $\cF$. 
\end{enumerate}  
Then the bundle $\O(1)$ on $\P(E)$ is not pseudo-effective, or $\cF$ is Hermitian flat. 
Moreover, if $X$ is projective then we can derive the same conclusion provided that
$\widehat{\Sym}^k\cF^\star$ is $\omega_X$-stable for some $k\geq  r$ instead of $\rm (ii)$.
\end{thm}
\medskip

\noindent Unlike the other articles dedicated to these topics, our methods here are
relying on new positivity results for direct images that we are now introducing.
\smallskip

\noindent Let $p:Y\to X$ be a holomorphic proper map between two
K\"ahler manifolds, and let $(L, h_L)$ be a Hermitian line bundle on the
total space $Y$. The metric $h_L$ could be singular, but we assume that
\begin{equation}\label{intr01}
\Theta_{h_L}(L)\geq 0
\end{equation}
in the sense of currents on $Y$. In this context it is then established that the direct image sheaf
\begin{equation}\label{intr02}
\cE:= p_\star\left((K_{Y/X}+ L)\otimes \mathcal I(h_L)\right)
\end{equation}
is semi-positively curved in the sense of Griffiths when endowed with the natural $L^2$ metric denoted by $h_{Y/X}$ cf. \cite{HPS}, \cite{PT}.
\medskip

\noindent The next result we are presenting here concerns the positivity of $\det \cE$ as follows.
\begin{thm}\label{thm0}
Let $p:Y\to X$ be a holomorphic surjective and proper map, where $X$ and $Y$ are K\"ahler manifolds. Moreover, we assume that $p$ is locally projective. Let $(L, h_L)$ be a Hermitian line bundle over $Y$ such that \eqref{intr01} holds true, and such that
\begin{equation}\label{intr030}  
\Theta_{h_L}(L)\wedge p^\star\omega_X^{n-1}\geq \ep_0 p^\star\omega_X^{n}
\end{equation}
where $\omega_X$ is a Hermitian metric on $X$, $n=\dim X$, and $\ep_0> 0$ is a positive real number. Then we have
\begin{equation}\label{intr040}
\Theta_{\det h_{Y/X}}\left(\det(\cE)\right)\wedge \omega_X^{n-1}\geq r\ep_0\omega_X^{n} 
\end{equation}
where $r$ is the rank of the direct image $\cE$.
\end{thm}
\medskip

\begin{remark}{\rm If we replace the hypothesis \eqref{intr030} with the
following
\begin{equation}\label{intr03007}
\Theta_{h_L}(L)\geq \ep_0p^\star\omega_X,
\end{equation}
then it is well-known that the curvature of $(\cE, h_{Y/X})$ is greater than
$\displaystyle \ep_0\omega_X\otimes \Id_{\End(\cE)}$. Thus our statement \ref{thm0}
can be seen as a stronger version of this result.
}
\end{remark}
\smallskip

\noindent Very roughly, the proof for Theorem \ref{thm0} goes 
as follows: by techniques due to \cite{Ber06}, \cite{BP08}, we show that it would be enough to show the inequality \eqref{intr040} 
in the particular case of map $p: U\times V\to U$ which is simply the projection
on the first factor, where $U$ is an open ball in some Euclidean space and $V$ is a Stein manifold. In this case the statement to prove looks quite different from  \eqref{intr040} and it will be established by a (long and) direct computation.
\smallskip

\noindent An immediate consequence of this statement is the following.
\begin{thm}\label{thm01}
  Let $p:Y\to X$ be a holomorphic surjective and proper map, where $X$ and $Y$ are K\"ahler manifolds. Moreover, we assume that $p$ is locally projective. Let $(L, h_L)$ be a Hermitian holomorphic line bundle over $Y$ such that
\begin{equation}\label{intr03}  
\Theta_{h_L}(L)\geq -Cp^\star\omega_X, \qquad \Theta_{h_L}(L)\wedge p^\star\omega_X^{n-1}\geq 0
\end{equation}
where $\omega_X$ is a Hermitian metric on $X$, $n=\dim X$, and $C> 0$ is a positive real number. Then we have
\begin{equation}\label{intr04}
\Theta_{\det h_{Y/X}}\left(\det(\cE)\right)\wedge \omega_X^{n-1}\geq 0, 
\end{equation}
where $r$ is the rank of the direct image $\cE$.
\end{thm}
\smallskip

\noindent A last result we mention here is the following statement, cf. Theorem \ref{bermganprop}, section 5.

\begin{thm}\label{thm01nancy}
Let $p:Y\to X$ be a holomorphic surjective and proper map, where $X$ and $Y$ are K\"ahler manifolds. 
Moreover, we assume that $p$ is locally projective. Let $L$ be a line bundle over the total space $Y$ endowed with a possible singular metric $h_L$ such that  $\Theta_{h_L} (L)\geq 0$, and 
\begin{equation}\label{cor030nancy}  
\Theta_{h_L}(L)\wedge p^\star\omega_X^{n-1}\geq \ep_0p^\star\omega_X^{n}
\end{equation}
where $\omega_X$ is a Hermitian metric on $X$, $n=\dim X$, and $\ep_0> 0$ is a positive real number. We assume that the space of fiberwise $\displaystyle L^{\frac{2}{m}}$ sections (with respect to $h_L$) of
$p_\star (m K_{Y/X} +L)$ is non zero. Then there exists a metric $h$ on the
bundle $m K_{Y/X} +L$ such that we have
\begin{equation}\label{fiberversion}
\Theta_{h} (m K_{Y/X} +L) \wedge p^\star \omega_X ^{n-1} \geq \ep_0 \cdot p^\star \omega_X ^n 
\end{equation}
in the sense of currents on $Y$.
\end{thm}
\smallskip

\begin{remark}{\rm 
We see that Theorem \ref{thm01nancy} has the same flavor as the usual results concerning direct images: the positivity of the curvature of $(L, h_L)$ induces similar properties for the twisted pluricanonical bundle 
$mK_{Y/X}+ L$. More precisely, if $\Theta_{h_L}(L)\geq 0$ then we already know that the bundle 
$m K_{Y/X} +L$
is pseudo-effective. We prove here that the additional positivity requirement \eqref{cor030nancy} is inherited by the twisted pluricanonical bundle of the map $p$.}
\end{remark}

\medskip

\noindent {\bf Organization of the paper.} The remaining part of this article will unfold as follows. In the first section we recall a few technical statements which are playing an important role
in our arguments. Then as a \emph{warm-up}, we prove Theorem \ref{thm1}
in the projective case; the main ideas are as follows.
Assume that $\O(1)$ is pseudo-effective. Then it admits a singular metric
$e^{-\varphi}$ whose curvature is semi-positive.
The stability condition (ii) implies that the multiplier ideal
$$\cI(e^{-(r+1-\ep)\varphi}|_{\P(E^\star _x)})$$ is trivial for any $\ep >0$, and for all $x\in X$
generic. In order to prove this we are using the positivity of direct images,
combined with the restriction theorem of Mehta-Ramanathan.
Now if the multiplier ideal above is trivial, then we can construct a singular Hermitian metric on $\cF^\star$ with positive curvature in the sense of Griffiths. Together with the fact that $c_1(\cF)= 0$, this implies that
$\cF$ is Hermitian flat by a result in \cite{CP17}.

\noindent The general case of Theorem \ref{thm1} is much more subtle,
since one has to find
an alternative argument in order to compensate the absence of
Mehta-Ramanathan theorem. It is at this point that Theorem \ref{thm01} comes into the picture. Thanks to this result we can still analyze the singularities of the
metric on $\O(1)$ as discussed above.
\smallskip

\noindent On the other hand, Theorem \ref{thm0} is interesting in its own right: it can be seen as very precise analysis of the positivity properties of the determinant of direct images. 
Further results and applications of this statement (and its proof)
are discussed in the last section of this article. We can interpret them
as a first step towards the analysis of the positivity of direct images on currents of bi-dimension $(1,1)$.
\medskip

%\begin{question}{\rm Another very interesting question in this framework is the following.
%    Consider the jet spaces tower $(X_k, \O_k(1))$ corresponding to a projective surface $X$. Assume that for some $k$, the tautological bundle $\O_k(1)$ is big. Theorem \ref{intr03} provides an important tool in studying the ``bigness'' of the determinant of the direct image $$\pi_{k\star}\left((K_{X_k/X_{k-1}}+ \O_k(m))\otimes \cI(h_k^{\otimes m})\right).$$
%The question is to make sure that the direct image above is not identically zero...     
%}\end{question}
\noindent {\bf Acknowledgement.} We would like to thank B. Berndtsson, A. Höring and M. Toma
for useful discussions and suggestions about the topics in this article.

\section{Stable reflexive sheaves and singular Hermitian metrics}

\noindent In this preliminary section we collect a few definitions and
properties concerning the
stability of reflexive sheaves on compact K\"ahler manifolds, as well as the notion of singular Hermitian metric on a vector bundle.

\subsection{Stable reflexive sheaves} We will mostly follow the article by Bando-Siu, cf.~\cite{BS94}, combined with important clarifications communicated to us by M.~Toma, cf. \cite{T19}.
\smallskip

\noindent Let $X$ be a compact K\"ahler manifold, and let $\cF$ be a reflexive
subsheaf. Then there exists a sequence of blow-up maps of non-singular centers
whose composition is denoted by
\begin{equation}\label{stab1}
\pi: \widehat X\to X
\end{equation}
such that the inverse image $\pi^\star (\cF)$ modulo torsion becomes a vector bundle denoted by $E\to \widehat X$ in \cite{BS94}.

\noindent Let $\omega$ be a K\"ahler form on $X$. Then the inverse image
$\wh \omega:= \pi^\star(\omega)$ \emph{is not} K\"ahler in general, but nevertheless
the class corresponding to $\wh \omega^{n-1}$ is movable. In this context, we have the following result cf. e.g. \cite{GKP1}.
\begin{lemme}\label{gkp}The sheaf $\cF$ is stable with respect to $\omega$ if and only if 
$E$ is stable with respect to $\wh \omega$.
\end{lemme}

\noindent For the basic facts concerning the stability with respect to a
movable class we refer to \cite{GKP1}.
\medskip

\noindent The important result of Bando-Siu states as follows.
\begin{thm}\cite{BS94}
Let $\cF$ be a reflexive sheaf on a compact K\"ahler manifold, which moreover is
stable with respect to a K\"ahler metric $\omega$. Then $\cF$ has an admissible Hermite-Einstein metric $h_{\cF}$.  
\end{thm}
\medskip

\noindent We do not recall here the notion of \emph{admissible metric}
of a vector bundle because we do not need it; we will rather work with its regularization defined as follows. Let
\begin{equation}\label{stab3}
K_{\wh X/X}= \sum_{i=1}^N e_i E_i  
\end{equation}
be the relative canonical divisor of the map $\pi$. 
Then there exist a set of non-singular representatives
$\theta_i\in c_1(E_i)$ such that the form
\begin{equation}\label{stab2}
\wh\omega_0:= \wh \omega- \ep_0\sum_{i=1}^N\ep_i\theta_i
\end{equation}
is K\"ahler for any positive and small enough $0<\ep_0\ll 1$,
provided that the coefficients $\ep_i$ are carefully chosen.
\smallskip

\noindent By Lemma \ref{gkp} the vector bundle $E$ is stable with respect to
$\wh\omega$. It is stated as a remark in \cite{BS94} that 
 $E$ is stable with respect to the perturbation $\wh\omega_0$ of $\wh\omega$,
As soon as the $\ep_0$ in \eqref{stab2} is small enough. We were not able to find a reference/proof for this assertion (which is most likely true) so we will rather
 use the following result established in the recent article \cite{T19}.
\begin{thm}\label{mtoma}\cite{T19}
The bundle $E$ is stable with respect to the degree function induced by the form
\begin{equation}\label{stab50}
(1-\ep)\wh\omega_0^{n-1}+ \ep\wh\omega^{n-1}
\end{equation}
for any $1\geq \ep\geq 0$, provided that $\ep_0\ll 1$. 
\end{thm}
\noindent It then follows that
there exists a metric $h_{E, \ep}$ on $E$ which verifies the Hermite-Einstein
condition with respect to $g_\ep^{n-1}$, where $g_\ep$ is the unique Gauduchon metric on $\wh X$ such that
\begin{equation}\label{stab51}
g_\ep^{n-1}= (1-\ep)\wh\omega_0^{n-1}+ \ep\wh\omega^{n-1}.
\end{equation}
Let us fix a coordinate system $(z_j)$ centered at some point $x_0$ of $\wh X$.
We denote by $G_\ep$ the matrix corresponding to 
$g_\ep$ and we define 
\begin{equation}\label{stab611}
(1-\ep)\wh\omega_0^{n-1}+ \ep\wh\omega^{n-1}|_U= (n-1)!(-1)^{\frac{(n-1)^2}{2}}\sum_{\alpha, \beta}(-1)^{\alpha+ \beta}\Omega_{\alpha \ol\beta, \ep}dz_\alpha\wedge d\ol z_{\beta}.
\end{equation}  
We have the formula $\displaystyle \Omega_\ep^tG_\ep= \det(G_\ep)I_n$ which
implies that we have the equality 
$\displaystyle
\det G_\ep= \left(\det\Omega_\ep\right)^{\frac{1}{n-1}}$.
We therefore obtain
\begin{equation}\label{stab512}
\int_{\wh X}dV_{g_\ep}\leq C
\end{equation}
for some constant $C> 0$ independent of $\ep$.  

\medskip

\noindent In conclusion, a reflexive sheaf $\cF$ which is stable with respect to a
K\"ahler metric $\omega$ has an admissible Hermite-Einstein metric $h_\cF$ which is limit of non-singular Hermite-Einstein metrics on a modification $\wh X$ of $X$
with respect to a Gauduchon metric $g_\ep$.

\subsection{Singular Hermitian Metrics and Positivity of Direct Images} 
For the convenience of the reader, we collect here a few basic notions and results concerning the positivity of
push-forward of relative canonical bundles, cf. \cite{PT}
\cite{HPS}, \cite{Pa15} and the references therein. As we will next see, the proof of the projective case of Theorem \ref{thm1} is
relying heavily on them.
\smallskip

\noindent The following notion appeared naturally in \cite{BP08}, and it was subsequently studied in \cite{Rau}, \cite{PT}, \cite{HPS}.
Let $E\to X$ be a holomorphic vector bundle of rank $r$ on a complex manifold $X$.
We denote by
\begin{equation}\label{sHm1}
H_{r}:=\{A=(a_{i\ol j})\}
\end{equation} the set of $r \times r$, semi-positive definite Hermitian matrices.
The manifold $X$ is endowed with the Lebesgue measure. We recall next the following notion.

\begin{defn}\label{sHm20}
A {\it singular Hermitian metric} $h$ on $E$ is given locally by a measurable map with values in 
$H_{r}$ such that 
\begin{equation}\label{sHm30}
0<\det h<+\infty
\end{equation} 
almost everywhere.
\end{defn}

\noindent Let $(E, h)$ be a vector bundle endowed with a 
singular Hermitian metric $h$. Given a {local section} $v$ of $E$, i.e. an element $v \in H^{0}(U,E)$ defined on some open subset $U\subset X$, the function 
$\displaystyle |v|_{h}^2 : U \to \R_{\ge 0}$ is measurable, given by 
\begin{equation}
|v|_{h}^{2} = {}^{t}v h \ol v=\sum h_{i\ol j}v^{i}\ol{v^{j}}
\end{equation} 
where $v={}^{t}(v^{1},\ldots,v^{r})$ is a column vector.
\medskip

\noindent Following \cite[p.\,357]{BP08}, we recall next the notion of positivity/negativity of a singular Hermitian vector bundle as follows.

\begin{defn}\label{curv}
Let $h$ be a singular Hermitian metric on $E$.
\smallskip

\begin{enumerate}

\item[(1)] The metric $h$ is {\it negatively curved} if the function
\begin{equation}
x\to \log |v|_{h, x}^{2}
\end{equation}
is psh for any local section $v$ of $E$.
\smallskip

\item[(2)] The metric $h$ is {\it positively curved} if the dual singular Hermitian metric $h^{\star}:={}^{t}h^{-1}$ on the dual vector bundle $E^{\star}$ is negatively curved.
\end{enumerate}
\end{defn}
\medskip

\noindent We recall the following important property of the notion \ref{curv}.

\begin{proposition}\label{regsh}\cite[Cor 2.8]{CP17} Let $(E, h)$ be a vector bundle endowed with a positively curved singular Hermitian metric $h$ on a compact manifold $X$. If the first Chern class $c_1(E)= 0$, then the metric $h$ is non-singular and 
the curvature of $(E, h)$ is equal to zero. 
\end{proposition}

\noindent If $E$ is replaced by a coherent, torsion-free sheaf $\cE$ then the conclusion of the statement
\ref{regsh} still holds provided that we restrict to the open subset on which
$\cE$ is a vector bundle.
\smallskip

\noindent The notion of positively curved singular Hermitian metric is the
natural property of push-forwards of relative canonical bundles as we see
from the next results.

\begin{thm}\label{push}\cite{PT, Pa15, HPS}. Let $p: Y\to X$ be an algebraic fiber space, and let $(L, h_L)$ be a Hermitian line bundle. The metric 
$h_L$ could be singular and the corresponding curvature current $\Theta(L, h_L)$ is semi-positive. Then the singular Hermitian  metric on the torsion-free push-forward sheaf
\begin{equation}\label{eq4} 
p_\star\left((K_{Y/X}+ L)\otimes \cI(h_L)\right)
\end{equation}
is positively curved in the sense of Griffiths.
If moreover $\cI(h_L |_{Y_x}) = \O_{Y_x} $ for a generic fiber $Y_x$, then $h_{Y/X}$ extends naturally as a singular Hermitian metric $p_\star\left((K_{Y/X}+ L) \right)$, which is also positively curved in the sense of Griffiths.
\end{thm}

\noindent We also recall the following variant of Theorem \ref{push}.

\begin{cor}\label{pushep}\cite[Lemma 5.25]{CP17} Let $p: Y\to X$ be an algebraic fiber space, and let $(L, h_L)$ be a Hermitian line bundle. The metric 
$h_L$ could be singular and the corresponding curvature current $\Theta(L, h_L) + C p^\star \omega_X$ is semi-positive for some $C >0$. 
Let $\xi$ be a local holomorphic section of the dual of $p_\star\left((K_{Y/X}+ L)\otimes \cI(h_L)\right)$. Then we have
\begin{equation}\label{sHm100}\sqrt{-1}\partial\overline{\partial} \log |\xi|^2 _{h_{Y/X}} \geq - C \omega_X.
\end{equation}  
\end{cor}

\noindent The property \eqref{sHm100} will be formally denoted by
\begin{equation}\label{sHm101}\Theta_{h_{Y/X}} (p_\star\left((K_{Y/X}+ L)\otimes \cI(h_L)\right)) \succeq -C \omega_X \otimes \Id.
\end{equation}
\medskip

\noindent We will equally need the following generalization of Proposition \ref{regsh} proved in \cite{CH19, CM19}.
\begin{thm}\label{weakpositivv}\cite{CH19,CM19}
Let $X$ be a projective manifold and $\cF$ be a reflexsive sheaf on $X$. Let $X_0$ be the locally free locus of $\cF$. If $c_1 (\cF)=0$ and for every $\ep >0$, there exsits a possible singular hermitian metric $h_\ep$ on $\cF |_{X_0}$ such that 
\begin{equation}\label{sHm1022}\Theta_{h_\ep} (\cF) \succeq -\ep  \omega_X \otimes \Id \qquad\text{on } X_0,
\end{equation}
then $\cF$ is a numerically flat vector bundle on $X$, namely $\cF$ is nef and $c_1 (\cF)=0$.
\end{thm}

\begin{remark}{\rm
    The theorem \ref{weakpositivv} is proved in \cite{CH19} for projective surfaces
    and it was very recently generalized to arbitrary dimension in \cite{CM19}.
    We recall briefly the idea of the proof and refer to  \cite{CM19} for a complete treatment.

We fix a polarisation $H$, and let $0\rightarrow \cG_1 \rightarrow \cG_2 \cdots \rightarrow \cG_m = \cF$
a be $H$-stable filtration. By an argument which parallels \cite[Thm 1.18]{DPS94}, we we show that every quotient $(\cG_{i+1} /\cG_i)^{\star\star}$ satisfies the same conditions like $\cF$, namely 
$$c_1 ((\cG_{i+1} /\cG_i)^{\star\star})=0 \qquad\text{and} \qquad \Theta_{h_{i,\ep}} ((\cG_{i+1} /\cG_i)^{\star\star}) \succeq -\ep  \omega_X \otimes \Id ,$$
where $h_{i,\ep}$ is a natural metric induced by $h_\ep$. 
Let $S$ be a surface defined by complete intersection of $H_1 \cap H_2 \cdots \cap H_{n-2}$, where $H_i$ is a generic hypersurface in the class of $|H|$. By applying \cite[Cor 2.12]{CH19} to 
$(\cG_{i+1} /\cG_i)^{\star\star} |_S$, we know that $(\cG_{i+1} /\cG_i)^{\star\star} |_S $ is hermitian flat. In particular, $c_2 ((\cG_{i+1} /\cG_i)^{\star\star} ) \wedge H^{n-2} =0$. Then \cite[Cor 3]{BS94} implies that  $(\cG_{i+1} /\cG_i)^{\star\star}$ is a hermitian flat vector bundle on $X$. As a consequence, we can prove that $\cF$
is an successive extension of hermitian flat vector bundles on $X$. It is thus numerically flat by \cite[Thm 1.18]{DPS94}. }
\end{remark}

\medskip

\section{Proof of the main result: the projective case}

\noindent In this section we will prove our main result under the assumption that $X$
is projective. Indeed, in this case all the necessary tools we will be using in our arguments are already available: we show that Theorem \ref{thm1}
follows from the positivity of direct images, combined with the existence of Hermite-Einstein metric on stable vector bundles.
\medskip

\noindent To make things precise, the result we establish in this first section is
the following.

\begin{thm}\label{thm1proj}
Let $(X, H)$ be a smooth polarized projective manifold. We consider a reflexive sheaf $\cF$ with the following properties.
\begin{enumerate}

\item[(i)] The first Chern class of $\cF$ is zero, i.e. $c_1(\cF)= 0$.
  \smallskip

\item[(ii)] The sheaf $\widehat{\Sym}^k\cF^\star$ is $H$-stable for some $k\geq r$, where $r$ is the rank of $\cF$.  
\end{enumerate}  
Then the bundle $\O(1)$ on $\P(E)$ is not pseudo-effective, or $\cF$ is Hermitian flat. \end{thm}
\noindent The bundle $E$ in (ii) is $\pi^\star(\cF)$ modulo torsion, cf. section 1.
\medskip

\noindent An immediate consequence of Theorem \ref{thm1proj} is the following statement.
\begin{cor}\label{coro1proj}
  Let $(X, H)$ be a smooth polarized projective manifold. We consider a
  saturated coherent subsheaf $\cF\subset T_X$ satisfying the properties in Theorem {\rm \ref{thm1proj}} and moreover
  we assume that $\cF$ is closed under the Lie bracket. Then $\cF$ is either Hermitian flat or all of its leaves are algebraic.
\end{cor}
\medskip  

\noindent We assume that the tautological bundle $\O(1)\to \P (E)$ is pseudo-effective.
Then we obtain a singular metric $h= e^{-\varphi}$ on $\O(1)$ whose curvature current $\Theta$ is positive,
$$\Theta\geq 0.$$

\noindent The main point in our arguments is the following statement concerning the singularities of the metric $h$. 

\begin{lemme}\label{trivia} We consider the map $p:\P(E)\to X$ obtained by composing the natural projection $\P(E)\to \wh X$ with $\pi$.
If $\Sym^k \cF^\star$ is $H$-stable for some $k\in \mathbb{N}$, then the multiplier sheaf corresponding to the restriction
\begin{equation}\label{eq1}
e^{-(k+1-\ep)\varphi}|_{\P({E})_x}
\end{equation} 
is trivial for any $\ep >0$. Here $x$ is any point in the complement of a countable union of proper analytic subsets of $X$, where we denote by $\P({E})_x$ the $p$--inverse image of the point $x\in X$.
\end{lemme}

\begin{remark}{\rm
In this lemma it is not necessary to require $k\geq r$.
Ideally, we want to consider the case $k=1$. The bundle
$\O(r+1)$ is endowed with the metric
$$e^{-(2-\ep)\varphi- (r-1+\ep)\varphi_{HE}}$$
where we assume that we are in the projective case and we restrict $\cF$
to a complete intersection curve $C$. We denote by $e^{-\varphi_{HE}}$ the
metric on $\O(1)$ induced by the flat Hermite-Einstein metric on $\cF|_C$. 
Of course, the direct image is just $\cF|_C$, but it is endowed with a metric depending on $e^{-\varphi}$. 
Since the
determinant is cohomologically trivial, the hope is that this should have
important consequences on the Hessian of $\varphi$ ``in the base directions'', cf. calculations section 4.}
\end{remark}

\noindent The rest of this section is divided into three parts. We will first show that Theorem \ref{thm1proj} follows from Lemma \ref{trivia}. Then we prove the lemma, and finally we give a quick argument for Corollary \ref{coro1proj}.

\subsection{Lemma \ref{trivia} implies Theorem \ref{thm1proj}}
\noindent Now we can prove Theorem \ref{thm1proj}: since $k \geq r$, thanks to the Lemma \ref{trivia}, we know that 
$\cI(e^{- (r+1-\ep)\varphi} |_{\P (E)_x}) = \mathcal{O}_{\P (E)_x}$ for a generic $x\in X$. The inclusion
\begin{equation}\label{eq10}
  p_\star\left(K_{\P(E)/X}\otimes \O(r+1)\otimes \cI(e^{- (r+1-\ep)\varphi})\right) \subset
  p_\star\left(K_{\P({E})/X}\otimes \O(r+1)\right) =\cF^\star
\end{equation} 
is thus generically isomorphic.

As $\O (1)$ is relatively ample, there exists a smooth metric $h_0$ on $\O (1)$ such that $\Theta_{h_0} \O (1) \geq - p^\star\omega_X$ for some K\"{a}hler metric $\omega_X$ on $X$.
Set $h_\ep := e^{- (r+1 -\ep) \varphi} \cdot h_0 ^{\ep}$. It is a metric on $\O (r+1)$ with $\Theta_{h_\ep} (\O (r+1)) \geq - \ep p^\star\omega_X$.
By Corollary \ref{pushep}, $h_\ep$ induces a metric $h_{\ep, \cF^\star}$ on $\cF^\star$ such that
$$i\Theta_{h_{\ep, \cF^\star}} (\cF^\star) \succeq -\ep \omega_X \otimes \Id .$$  
Together with the fact that $c_1 (\cF^\star)=0$, it follows that $\cF^\star$ is numerically flat by Theorem \ref{weakpositivv}.
Note that the stablility of $\widehat{\Sym}^k \cF^\star$ implies that $\cF^\star$ is stable. $\cF^\star$ is thus hermitian flat by using \cite[Thm 1.18]{DPS94}.

\subsection{Proof of Lemma \ref{trivia}}
By restriction theorem of Mehta-Ramanathan, the bundles
\begin{equation}\label{eq2}
\Sym^k \cF^\star|_C
\end{equation}
are still be stable provided that $C$ is a generic complete intersection of $n-1$ divisors in a
large enough multiple of $H$. 

The curves ``$C$" for which the restriction \eqref{eq2} is stable cover a non-empty 
Zariski open subset of $X$. 
We assume that our point $x$ belongs to the union of these curves, so $x\in C$ for some complete intersection curve
which we fix for the rest of the proof. 

Let $h_{C}$ be the Hermite-Einstein metric on the bundle $\displaystyle \cF|_C$. It is of course smooth, and thanks to the first 
hypothesis (i) the curvature tensor $\Theta(\cF|_C, h_C)$ is identically equal to zero. In particular, $h_C$ 
induces a smooth metric on $\O(1)$ whose curvature is semi-positive. Roughly speaking, this is so because the curvature of $\O(1)$ at a point $(y, [v])$ has two components: 
the Fubini-Study form in the directions of the fibers of the map
\begin{equation}\label{eq3}
\pi: \P(\cF|_{C})\to C
\end{equation}
plus the curvature $\Theta(\cF|_C, h_C)$ evaluated in the direction $v$.

\noindent We assume that $x$ is such that the multiplier ideal $\displaystyle \cI(e^{-(k+1-\ep)\varphi}|_{\P({\cF_x})})$ coincides with 
$\displaystyle \cI(e^{-(k+1-\ep)\varphi})|_{\P({\cF_x})}$. Again, this is a genericness requirement which holds true in the complement of a countable union of points of $C$. 

Consider the bundle 
\begin{equation}\label{eq5} 
K_{\P({\cF})/C}+ L.
\end{equation}
We endow $L:= \O(k+r)$ with a the metric whose local weights are 
\begin{equation}\label{eq6}
\psi_L:= (k+1-\ep)\varphi+ (r-1+\ep)\varphi_C;
\end{equation}
the corresponding curvature curent is semi-positive. By standard $L^2$ estimates cf. Lemma \ref{l2est}, the vector space
\begin{equation}\label{eq7} 
H^0\left(\P({\cF_x}), \O(K_{\P({\cF})/C} +L)\otimes \cI(e^{-\psi_L})\right) = H^0\left(\P({\cF_x}), \O( k)\otimes \cI( e^{-(k+1 -\ep)\varphi} |_{\P({\cF_x})})\right) 
\end{equation}
is non-trivial. Moreover, it does not coincides with the space of global sections of $\O(k)|_{\P({\cF_x})}$ if the multiplier ideal sheaf of
the metric $e^{-\psi_L}$ is non-trivial.

By Theorem \ref{push} we infer that the direct image
\begin{equation}\label{eq8}
\cG := \pi_\star\left( (K_{\P({\cF})/C}+ L)\otimes \cI(e^{-\psi_L})\right)
\end{equation}
is semi-positively curved. In particular, its degree with respect to $C$ is semi-positive. 

\noindent On the other hand, we have 
\begin{equation}\label{eq9}
\cG \subset \pi_\star\left(K_{\P({\cF})/C}+ L \right) =\Sym^k \cF^\star
\end{equation}
(modulo a topologically trivial line bundle) and as explained before, the inclusion is strict. Hence the degree of $\cG$ is \emph{strictly smaller} than the 
degree of $\Sym^k\cF^\star$, which is zero. 
The lemma is proved in the projective case.
\smallskip

\begin{question}{\rm
    We see that the stability condition allow us to deduce regularity
    properties of the metric $e^{-\varphi}$ in case $c_1(\cF)= 0$. Can one
formulate (and eventually prove) similar results if $c_1(\cF)$ contains a negative representative? 
}\end{question}

\subsection{Proof of Corollary \ref{coro1proj}} The statement \ref{coro1proj}
follows immediately as combination of Theorem \ref{thm1proj} and the following
algebraicity criteria which is a direct consequence of
\cite{CP19} and for which we refer to cf. \cite{Dr18}.

\begin{thm} Let $X$ be a projective manifold and let $\cF\subset T_X$ be a holomorphic foliation. If $\cF$ is not algebraic, then the tautological line bundle
$\O(1)$ over $\P(E)$ is pseudo-effective. Here we denote by $E$ any desingularization of the sheaf $\cF$.  
\end{thm}

\noindent Indeed, this result does not appears explicitly in \cite{CP19}
but it is quickly deduced from the proof of \cite[Thm 1.1]{CP19}.
More precisely this is proved in section 4.1, pages 14-18 of \emph{loc.cit.}
  
\section{Proof of the main result: the general case}

\noindent We establish in the next subsection Theorem \ref{thm0}. This statement can be seen as a subharmonicity property of direct images and it represents the new technical result in the present article. Once this is done we show that the K\"ahler version of Lemma \ref{trivia} follows. Further results based on Theorem \ref{thm0} will be given in the next section.
\smallskip

\noindent A good point to start with is the remark that if $h_L$ is non-singular and if $p$ is a submersion,
then \eqref{intr040} follows 
directly from the curvature formula \cite[(4.8)]{Ber09}.

\noindent Indeed, let $x\in X$ be an arbitrary point
and let $\displaystyle u\in H^0 (Y_x, K_{Y_x} +L|_{Y_x})$ be a holomorphic section. 
Since $h_L$ and $p$ are smooth, we can find a section
$$\wh u\in H^0 \left(U, p_\star(K_{Y/X}+ L)\right)$$
extending $u$ as in \emph{loc. cit.}, \cite[Prop 4.2]{Ber09}. Here we denote by $U$ a
small coordinate set centered at $x$.

\noindent The equality \cite[(4.8)]{Ber09} implies that we have
$$\left\langle\Theta_{h_{Y/X}} \left(p_\star(K_{Y/X}+ L)\right)u,u\right\rangle \wedge \omega_X ^{n-1}\geq c_{n}\int_{Y_x} \widehat{u}\wedge \overline{\widehat{u}} \wedge \Theta_{h_L}(L)
\wedge p^\star (\omega_X ^{n-1}) e^{-\phi_L} .$$
Since by hypothesis \eqref{intr030} we have,
$$ \Theta_{h_L}\wedge p^\star \omega_X ^{n-1}\geq \ep_0 p^\star \omega_X ^{n}$$
it follows that
$$\left\langle \Theta_{h_{Y/X}} \left(p_\star(K_{Y/X}+ L)\right)u,u\right\rangle \wedge \omega_X ^{n-1} \geq \ep_0 \|u\|^2 _{h_{Y/X}} \omega_X ^{n} $$
\noindent at $x$. We obtain thus \eqref{intr040} if $h_L$ is smooth by taking the trace.

\subsection{A subharmonicity property of direct images}
In general case the metric $h_L$ is not necessary smooth, but nevertheless the function
$\displaystyle \log|\sigma|_{h_{Y/X}^\star}^2$ is still psh for any holomorphic
section $\sigma$ of $\cE^\star$ defined on some open subset $U$ of $X$. This is the content of Theorem \ref{push} we have recalled in Section 2.
\medskip

\noindent The next step is to show that
under the assumptions of Theorem \ref{thm0} the following holds.

\begin{claim}\label{mainc} Let $\sigma$ be a local holomorphic section of the dual bundle $\cE^\star$. Then we have
\begin{equation}\label{sh1}
dd^c\log|\sigma|_{h_{Y/X}^\star}^2\wedge \omega_X^{n-1}\geq \ep_0\omega_X^n
\end{equation}
in the sense of currents on $X$.
\end{claim}
\medskip

\noindent Prior to explaining the proof of this claim, we see here that it implies
the inequality \eqref{intr040}. To this end, we are using a trick from the article \cite{HPS}, which
consists in considering the relative adjoint bundle
\begin{equation}\label{sh2}
K_{\P(\cE^\star)/X_0}+ \O(r)
\end{equation}
on $\displaystyle \P(\cE^\star)$, where $X_0\subset X$ is the maximal subset
of $X$ such that the restriction $\displaystyle \cE|_{X_0}$ is a vector bundle.
The bundle $\O(1)$ is endowed with the metric $h$ induced from $h_{Y/X}$. Then the relation \eqref{sh1} implies that we have
\begin{equation}\label{sh3}
\Theta_h(\O(1))\wedge \pi^\star\omega_X^{n-1}\geq \ep_0 \pi^\star\omega_X^{n}.
\end{equation}
Now we apply the Claim one more time, for the direct image sheaf
\begin{equation}\label{sh4}
\pi_\star\left(K_{\P(\cE^\star)/X}+ \O(r)\right)
\end{equation}
which coincides with $\displaystyle \det \cE|_{X_0}$.
Since $h$ is induced by the metric $h_{Y/X}$, it follows by the definition of a singular metric the corresponding multiplier sheaf is trivial over a large open subset of $\P(\cE^\star)$.   
Then we have
\begin{equation}\label{sh5}
\chi_{X_0}\Theta_{\det h_{Y/X}}( \det \cE)\wedge \omega_X^{n-1}\geq r\ep_0\omega_X^{n}.
\end{equation}
in the sense of currents on $X_0$ as consequence of \eqref{sh1}, where $\chi_{X_0}$ is the characteristic function
of the set $X_0\subset X$. Indeed this follows from the fact that 
\begin{equation}\label{sh2002}
T:= \Theta_{\det h_{Y/X}}( \det \cE)\wedge \omega_X^{n-1}
\end{equation}
is a closed positive current on $X$, and we clearly have  
\begin{equation}\label{sh2003}
T\geq \lim_{k\to\infty}v_k\Theta_{\det h_{Y/X}}( \det \cE)\wedge \omega_X^{n-1}
\end{equation}
for a sequence $(v_k)$ converging to $\chi_{X_0}$, cf. \cite{agbook}. For each $k$ we have
\begin{equation}\label{sh2004}
v_k\Theta_{\det h_{Y/X}}( \det \cE)\wedge \omega_X^{n-1}\geq v_k r\ep_0\omega_X^{n}
\end{equation}
as currents of bi-degree $(n,n)$ on $X$. Summing up, we have
\begin{equation}\label{sh2005}
T\geq \lim_{k\to \infty}v_k r\ep_0\omega_X^{n}
\end{equation}
and this is equivalent to the inequality \eqref{intr040}.
\smallskip

\begin{remark}{\rm
It is absolutely crucial that the lower bound \eqref{sh1} is the
same as the lower bound of the trace of the curvature of $(L, h_L)$ we have by hypothesis in Theorem \ref{thm01}.} 
\end{remark}

\noindent We detail now the proof of Claim \ref{mainc}.
An important point in our argument here is that the function
$\displaystyle \log|\sigma|_{h_{Y/X}^\star}^2$ is \emph{already} known to be
psh.
It would be therefore enough to show that the inequality
\eqref{sh1} holds true in the complement of an analytic subset of $X$.
In particular, we can assume that $p$ is smooth, and that $\cE$ is a trivial vector bundle of rank $r$.

\subsubsection{Reduction to a local statement} We invoke here the regularization arguments in \cite{BP08} in order to reduce our claim to a subharmonicity property of
fiberwise Bergman kernels for Stein submanifolds in $\CC^n$.

 Let $s_1, s_2, \cdots, s_r$
be a set of sections of the fibration $p: p^{-1} (U)\rightarrow U$ and let $a_i$ be local sections of $\Lambda^{n}T_{Y/X}\otimes L^{-1}$ defined locally near the image of $s_i$.
We define the Bergman kernel-type function on $U$ as follows
\begin{equation}\label{sh6}
  B \langle a, a\rangle :=\sum_{j,k} a_j \overline{a}_k B (s_j , s_k) ,
\end{equation}
where 
\begin{equation}\label{sh7}
  B (s_j (y), s_k (y) ) =\sum_i u_i (s_j (y)) \otimes \overline{u}_i (s_k (y)),
\end{equation}  
and $\{u_1, \dots, u_r\}$ is an orthogonormal basis of
$H^0\left(X_y,  (K_{X/Y}+ L)\otimes \cI(h_{\phi_L})\right)$. Also, we are using the same notation $a_i$ for the section
$a_i$ evaluated at $s_i$. Therefore the expression \eqref{sh6} is indeed a function on $U$.
\smallskip

\noindent The link between \eqref{sh6} and our problem is as follows.
We consider the expression
\begin{equation}\label{sh80}
\xi_y:= \sum a_j(s_j(y)){\rm ev}_{s_j(y)}
\end{equation}
which is a local holomorphic section of the dual bundle $\cE^\star$.
Conversely, as observed in \cite{BP08} any local section of $\cE^\star$ can be obtained in this manner.
Moreover, the norm of the section \eqref{sh80}
with respect to the metric $h_{Y/X}^\star$ is precisely $B \langle a, a\rangle$.
This is very important, because it shows that $B \langle a, a\rangle$ is an
\emph{extremal function}.
\smallskip

\noindent It would be therefore
sufficient to show that for every such $a$ we have 
\begin{equation}\label{equberg}
\ddc \log B \langle a, a\rangle \wedge \omega_X ^{n-1} \geq \ep_0 \omega_X ^n 
\end{equation}
on $U$. Actually, thanks to a standard trick which we recall in Lemma \ref{LGG}
below it is enough to show the -equivalent- inequality
\begin{equation}\label{equberg1}
\ddc B \langle a, a\rangle \wedge \omega_X ^{n-1} \geq \ep_0 B \langle a, a\rangle \omega_X ^n 
\end{equation}
\medskip

\noindent It is at this point that we need our map $p$ to be locally projective.
In this case we can use the regularization procedure in \cite{Ber06}, \cite{BP08}
and reduce our problem (i.e. the inequality \eqref{equberg1} above) to a local
situation in which we can solve it by a direct computation. We will therefore recall the main steps of the relevant part of \cite{BP08} and explain the way in which this is implemented in our case.
\smallskip

\noindent $\bullet$ According to the hypothesis, given any point $t_0\in X$ there exists an open subset $U\subset X$ containing $x_0$ such that 
\begin{equation}\label{reg4}
D:= p^{-1}(U)\setminus H
\end{equation}
is a Stein manifold,
where $H$ is a hypersurface. We can therefore replace our initial map with a Stein fibration
\begin{equation}\label{reg8}
p: D\to U.
\end{equation}
The
corresponding Bergman kernel $\displaystyle B\langle a, a\rangle$
is
defined as in \eqref{sh6} except that we replace the finite dimensional basis
$u_1,\dots, u_r$ with a Hilbert basis of the space of holomorphic $L^2$ top forms on $D_t:= D\cap p^{-1}(t)$ with respect to the $L^2$ norm
\begin{equation}\label{reg7}
\Vert u\Vert^2_{t}:= \int_{D_t}c_nu\wedge \ol ue^{-\varphi_L}.
\end{equation}
\smallskip 

\noindent $\bullet$ Let $\rho$ be a psh exhaustion function for the Stein manifold $D$. For $c$ large enough, the image of the sections $s_i$ will be contained
in $D^c:= (\rho < c)$. Let $B_c\langle a, a\rangle$ be the function \eqref{sh6}
associated to the domain $D^c$. By the extremal characterization of $B\langle a, a\rangle$ and $B_c\langle a, a\rangle$ respectively, we see that it would be enough to show that we have
\begin{equation}\label{reg6}
\ddc B_c \langle a, a\rangle \wedge \omega_X ^{n-1} \geq \ep_0 B_c \langle a, a\rangle \omega_X ^n 
\end{equation}
in the sense of currents on $D^c$. The inequality \eqref{equberg1} would
follow
as $c\to \infty$. 

\smallskip

\noindent $\bullet$ Let $c$ be
a regular value for the function $\rho$. There exists a biholomorphism
\begin{equation}\label{reg9}
\Phi: D_{t_0}^{c+ 1/2}\times U\to \Omega
\end{equation}
 such that  
\begin{equation}\label{reg10}
D_t^{c}\subset \Phi\big(D_{t_0}^{c+ 1/2}\times \{t\}\big)\subset D_t^{c+ 1}.   
\end{equation}
This map is obtained as usual, by considering the flow associated to the
holomophic lifting of vector fields of $U$ (we might be forced to shrink $U$,
but this is fine since the estimates we have to prove are independent of the size of this set).
\smallskip

\noindent We can assume that the bundle $L$ is trivial when restricted to $D$, and that the weight of the metric $h_L$ is given by the psh function $\varphi$.
Via the map \eqref{reg9}, the curvature hypothesis \eqref{intr030} becomes
\begin{equation}\label{reg11}
\ddc\phi\wedge \pi^\star\omega^{n-1}\geq \ep_0 \pi^\star\omega^{n}
\end{equation}
where
\begin{equation}\label{reg12}
\phi:= \varphi\circ \Phi,\qquad \omega:= \omega_X|_U. 
\end{equation}
and $\pi$ is the projection on the second factor of $D_{t_0}^{c+ 1/2}\times U$.
\smallskip

\noindent Let $\displaystyle \big(B(x_i, \ep)\big)_{i=1,\dots N_\ep}$ be a covering of the base $X$ with balls of radius $\ep$. On each such ball we certainly have
\begin{equation}\label{reg1}
(1-\delta_\ep)\omega_i\leq \omega_X\leq (1+ \delta_\ep)\omega_i
\end{equation}
where $\omega_i$ is a flat metric on $B(x_i, \ep)$, and
$\delta_\ep\to 0$ as $\ep\to 0$.

\noindent Then we have
\begin{equation}\label{reg2}
dd^c\phi\wedge \pi^\star\omega_i^{n-1}\geq \ep_0\frac{(1-\delta_\ep)^{n-1}}{(1+\delta_\ep)^{n}}\pi^\star\omega_i^{n}
\end{equation}
for each index $i$ such that $x_i\in U$ (so we assume that $\Phi$ is defined on a domain slightly bigger, which is again harmless).
\smallskip

\noindent
Now the point is that the inequality \eqref{reg2} behaves well with respect to convolutions on the inverse image of
$B(x_i, \ep)$ --since $\omega_i$ is flat-- in the sense that we have
\begin{equation}\label{reg131}
dd^c\phi_\delta\wedge \pi^\star\omega_i^{n-1}\geq \ep_0\frac{(1-\delta_\ep)^{n-1}}{(1+\delta_\ep)^{n}}\pi^\star\omega_i^{n}
\end{equation}
for a monotonic sequence of smooth psh functions $(\phi_\delta)$ converging to $\phi$. This can be seen e.g. as follows. We first use the fact that the Stein manifold
$D_{t_0}^{c+ 1/2}$ can be embedded in an Euclidean space $\CC^N$ (in our case this is much simpler, since it is a domain in a projective manifold). The image of the
embedding
\begin{equation}\label{reg133}
D_{t_0}^{c+ 1/2}\to \CC^N
\end{equation}
has a Stein neighborhood $W$ by a result of Siu \cite{}, and moreover we have a holomorphic retract $\mu: W\to D_{t_0}^{c+ 1/2}$. We can regularize the function 
\begin{equation}\label{reg134}
\wt \phi: U\times W\to \R\cup\{-\infty\} \qquad \wt \phi(t, w) := \phi\big(t, \mu(w)\big)
\end{equation}
by the usual convolution kernel, and $\phi_\delta$ is the restriction of the convolution to $U\times D_{t_0}^{c+ 1/2}$. Given this explicit definition of $\phi_\delta$,
the inequality \eqref{reg131} becomes obvious.

\smallskip

\noindent $\bullet$ It is proved in \cite{BP08} that there exists a sequence
of smooth psh functions $\psi_j$ such that if we denote by
$\displaystyle B_{j,\delta}\langle a, a\rangle$ the Bergman kernel-type function induced fiberwise by $e^{-\varphi_\delta- \psi_j}$, then the limit
\begin{equation}\label{reg14}
\lim_jB_{j,\delta}\langle a, a\rangle|_{D_{t_0}^{c+ 1/2}\times \{t\}}
\end{equation}
is \emph{equal} to the Bergman kernel on the fiber $D_t$ of our initial map $p$
with respect to the weight $\phi_\delta\circ \Phi^{-1}$. The role of the sequence
$(\psi_j)$ is thus to ``erase'' the difference between the domains in \eqref{reg10}.  This is first proved for points $t$ in the complement of measure zero set,
and then in general, cf. pages 355-356 in \cite{BP08}. A last observation at this point is that we also have
\begin{equation}\label{reg15}
dd^c(\phi_\delta+ \psi_j)\wedge \pi^\star\omega_i^{n-1}\geq \ep_0\frac{(1-\delta_\ep)^{n-1}}{(1+\delta_\ep)^{n}}\pi^\star\omega_i^{n}
\end{equation}
for every $j$, since $\psi_j$ is psh.  

\smallskip

\noindent $\bullet$ We show in the next subsection that we have
\begin{equation}\label{equberg2}
\ddc B_{j,\delta} \langle a, a\rangle \wedge \omega_i ^{n-1} \geq \ep_0 \frac{(1-\delta_\ep)^{n-1}}{(1+\delta_\ep)^{n}}B_{j,\delta}\langle a, a\rangle \omega_i ^n 
\end{equation}
on $B(x_i, \ep)$. By using \eqref{reg1} and the fact that $B \langle a, a\rangle $ is psh, we infer that we have 
\begin{equation}\label{reg3}
\ddc B \langle a, a\rangle \wedge \omega_X ^{n-1} \geq \ep_0 \frac{(1-\delta_\ep)^{2n-2}}{(1+\delta_\ep)^{2n}}B \langle a, a\rangle \omega_X ^n   
\end{equation}
and we are done, provided that we are able to establish \eqref{equberg2}. This will be done in the next subsection.  
\medskip

\subsubsection{The local computation} As consequence of the discussion in the
previous subsection, it is enough to consider the following set-up:
\begin{enumerate}
\smallskip
  
\item[(1)] The map $p: U\times V\to U$ is simply the projection on the $1^{\rm st}$ factor. Here $U$ is the unit ball in $\CC^n$ and $V$ is a Stein manifold.
\smallskip
  
\item[(2)] We denote by $t_1,\dots, t_n$ the coordinate functions on $U$, and
we use the notation $z_1,\dots, z_d$ for coordinates on $V$ at some point.
\smallskip
  
\item[(3)] We have a smooth psh function $\varphi$ on $U\times V$ such that
\begin{equation}\label{sh9}\nonumber
dd^c\varphi\wedge p^\star\omega^{n-1}\geq \ep_0p^\star\omega^{n}
\end{equation}
point-wise on $U\times V$.
\smallskip
  
\item[(4)] For each $t\in U$ we denote by $B_t$ the Bergman kernel on
\begin{equation}\label{sh10}\nonumber
  \left(\{t\}\times V, e^{-\varphi(t, \cdot)}\right) 
\end{equation}
corresponding to $(d, 0)$ holomorphic forms.
\smallskip
  
\item[(5)] Let $s_i: U\to V$ be a set of holomorphic functions (i.e. sections of the projection map $p$), where $i=1,\dots r$. We also consider the holomorphic sections $a_i$ of $\displaystyle s_i^\star\Lambda^dT_V$ and we define
\begin{equation}\label{sh11}
B_t\langle a, a\rangle := \sum_{k, p=1}^ra_k(t)\ol{a_p(t)}B_t\big(s_k(t), s_p(t)\big)
\end{equation}
\end{enumerate}
\smallskip

\noindent Now we will show that we have
\begin{equation}\label{sh12}
dd^cB_t\langle a, a\rangle \wedge \omega^{n-1}\geq \ep_0B_t\langle a, a\rangle \omega^n
\end{equation}
at each point $t\in U$. This would be indeed sufficient, thanks to the following standard fact.

\begin{lemme}\label{LGG}
Assume that \eqref{sh12} holds true. Then we have
\begin{equation}\label{sh13}
dd^c\log B_t\langle a, a\rangle \wedge \omega^{n-1}\geq \ep_0\omega^n
\end{equation}  
\end{lemme}
\begin{proof}
The argument is well-know in the case of positive psh functions, cf.
e.g. \cite{LG}, Lemma 3.46. Strictly the same proof goes through in our situation,
as we will see in what follows.

\noindent The main observation is that if we replace the function $\varphi$ with
\begin{equation}\label{sh013}
\varphi_\lambda:= \varphi+ \Re\left(\sum_{i=1}^n\alpha_it_i\right) 
\end{equation}
then the hypothesis (3) above is still verified. The corresponding fiberwise Bergman
kernel becomes
\begin{equation}\label{sh014}
e^{\Re\left(\sum_{i=1}^n\alpha_it_i\right)}B_t\langle a, a\rangle.
\end{equation}
Now let $t_0\in U$ be an arbitrary point. We assume that the coordinates $(t_i)$ are such that $\omega_{t_0}$ is the flat metric. Since the inequality
\eqref{sh12} is true for the Bergman kernel \eqref{sh014}, we infer that we have
\begin{equation}\label{sh015}
\sum_i \frac{\partial ^2B}{\partial t_i\partial\ol t_i}+ \frac{1}{2}\left(\alpha_i\frac{\partial B}{\partial \ol t_i}+ \ol\alpha_i\frac{\partial B}{\partial t_i}\right)+ \frac{1}{4}|\alpha_i|^2B\geq n\ep_0B
\end{equation}
where we denote by $B:= B_t\langle a, a\rangle$, and the quantities above are evaluated at $t= t_0$. The inequality \eqref{sh015} holds true for any choice of the coefficients $\alpha_i$. We then choose
\begin{equation}\label{sh016}
\alpha_i:= -\frac{2}{B(t_0)}\frac{\partial B}{\partial t_i}(t_0)
\end{equation}
and the lemma follows.
\end{proof}
\medskip

\noindent The inequality \eqref{sh13} will be established by a direct
computation detailed along the next lines. The case $n=1$ was treated by \cite{Ber06},
and our arguments below represent a generalization of his approach.
\smallskip

\noindent By the reproducing property of Bergman kernels, we have
\begin{equation}\label{sh15}
  B_t\langle a, a\rangle= \sum_{i, j} a_i(t)\ol{a_j(t)}\int_V B_t(\xi, s_i(t))
  \ol{B_t(\xi, s_j(t))}e^{-\varphi(t, \xi)}
\end{equation}
  
\noindent We first take the anti-holomorphic derivative with respect to
$t_\alpha$.

\begin{align}\label{sh14}
\frac{\partial}{\partial \ol t_\alpha } B_t\langle a, a\rangle= & \sum_{i, j} a_i(t)\ol{\frac{\partial a_j(t)}{\partial t_\alpha}}\int_V B_t(\xi, s_i(t))
  \ol{B_t(\xi, s_j(t))}e^{-\varphi(t, \xi)}\\
+ & \sum_{i, j} a_i(t)\ol{a_j(t)}\int_V \frac{\partial B_t}{\partial \ol t_\alpha}(\xi, s_i(t))
    \ol{B_t(\xi, s_j(t))}e^{-\varphi(t, \xi)}\nonumber\\
+ & \sum_{i, j} a_i(t)\ol{a_j(t)}\int_V \frac{\partial B_t}{\partial \ol w_\gamma}(\xi, s_i(t))\ol{\frac{\partial s_i^\gamma}{\partial t_\alpha}}
  \ol{B_t(\xi, s_j(t))}e^{-\varphi(t, \xi)}  \nonumber\\
+ & \sum_{i, j} a_i(t)\ol{a_j(t)}\int_V B_t(\xi, s_i(t))
    \ol{\frac{\partial^\varphi B_t}{\partial t_\alpha}(\xi, s_j(t))}e^{-\varphi(t, \xi)}.\nonumber\\
=: & I_1+ I_2+ I_3+ I_4.\nonumber\\   
\nonumber \end{align}
Given any holomorphic top form $h$ which is $L^2$ on $V$ we have
\begin{equation}\label{sh16}
h(s_j(t))= \int_Vh(\xi)\ol{B_t(\xi, s_j(t))}e^{-\varphi(t, \xi)}.
\end{equation}
We apply the operator $\displaystyle \frac{\partial}{\partial \ol t_\alpha}$
to \eqref{sh16} and it follows that we have
\begin{equation}\label{sh17}
0= \int_V h(\xi)
  \ol{\frac{\partial^\varphi B_t}{\partial t_\alpha}(\xi, s_j(t))}e^{-\varphi(t, \xi)}.
\end{equation}
Therefore the last term $I_4$ of \eqref{sh14} is equal to zero.
\medskip

\noindent We evaluate next the quantity
\begin{equation}\label{sh18}
\sum_{\alpha, \beta}\omega^{\ol \alpha\beta}\frac{\partial^2}{\partial t_\beta\partial \ol t_\alpha } B_t\langle a, a\rangle
\end{equation}
where $\displaystyle \left(\omega^{\ol\alpha \beta}\right)$ be the inverse
of the coefficients of the metric $\omega$ at some point $t_0\in U$.
\smallskip

\noindent The $\displaystyle \frac{\partial}{\partial t_\beta}$ of the term $I_1$ gives
\begin{align}\label{sh19}
\frac{\partial I_1}{\partial t_\beta}= & \sum_{i, j} \frac{\partial a_i(t)}{\partial t_\beta}\ol{\frac{\partial a_j(t)}{\partial t_\alpha}}\int_V B_t(\xi, s_i(t))\ol{B_t(\xi, s_j(t))}e^{-\varphi(t, \xi)}\\
  + & \sum_{i, j} a_i(t)\ol{\frac{\partial a_j(t)}{\partial t_\alpha}}
 \int_V \frac{\partial^\varphi B_t}{\partial t_\beta}(\xi, s_i(t))
 \ol{B_t(\xi, s_j(t))}e^{-\varphi(t, \xi)} \nonumber\\ 
+ & \sum_{i, j} a_i(t)\ol{\frac{\partial a_j(t)}{\partial t_\alpha}}\int_V B_t(\xi, s_i(t))
  \ol{\frac{\partial B_t}{\partial \ol t_\beta}(\xi, s_j(t))}e^{-\varphi(t, \xi)}
\nonumber\\
+ & \sum_{i, j} a_i(t)\ol{\frac{\partial a_j(t)}{\partial t_\alpha}}\int_V B_t(\xi, s_i(t))
  \ol{\frac{\partial B_t}{\partial \ol w_\gamma}(\xi, s_j(t))\ol{\frac{\partial s_j^\gamma}{\partial t_\beta}}}e^{-\varphi(t, \xi)}
\nonumber\\  
\nonumber \end{align}

\noindent We compute similar derivative for the other terms:
\begin{align}\label{sh20}
\frac{\partial I_2}{\partial t_\beta}= & \sum_{i, j} \frac{\partial a_i(t)}{\partial t_\beta}\ol{a_j(t)}\int_V \frac{\partial B_t}{\partial \ol t_\alpha}(\xi, s_i(t))
\ol{B_t(\xi, s_j(t))}e^{-\varphi(t, \xi)}\\
+ & \sum_{i, j} a_i(t)\ol{a_j(t)}\int_V \frac{\partial^\varphi}{\partial t_\beta}\Big(\frac{\partial B_t}{\partial \ol t_\alpha}\Big)(\xi, s_i(t))
    \ol{B_t(\xi, s_j(t))}e^{-\varphi(t, \xi)} \nonumber\\
+ & \sum_{i, j} a_i(t)\ol{a_j(t)}\int_V \frac{\partial B_t}{\partial \ol t_\alpha}(\xi, s_i(t))
    \ol{\frac{\partial B_t}{\partial \ol t_\beta}(\xi, s_j(t))}e^{-\varphi(t, \xi)}
    \nonumber\\
+ & \sum_{i, j} a_i(t)\ol{a_j(t)}\int_V \frac{\partial B_t}{\partial \ol t_\alpha}(\xi, s_i(t))
    \ol{\frac{\partial B_t}{\partial \ol w_\gamma}\ol{\frac{\partial s_j^\gamma}{\partial t_\beta}}(\xi, s_j(t))}e^{-\varphi(t, \xi)}\nonumber\\
\nonumber
\end{align}
and 
\begin{align}\label{sh21}
\frac{\partial I_3}{\partial t_\beta}= & \sum_{i, j} \frac{\partial a_i(t)}{\partial t_\beta}\ol{a_j(t)}\int_V \frac{\partial B_t}{\partial \ol w_\gamma}(\xi, s_i(t))\ol{\frac{\partial s_i^\gamma}{\partial t_\alpha}}
  \ol{B_t(\xi, s_j(t))}e^{-\varphi(t, \xi)}\\
  + & \sum_{i, j} a_i(t)\ol{a_j(t)}\int_V \frac{\partial^\varphi}{\partial t_\beta}
      \Big(\frac{\partial B_t}{\partial \ol w_\gamma}\Big)(\xi, s_i(t))\ol{\frac{\partial s_i^\gamma}{\partial t_\alpha}}
    \ol{B_t(\xi, s_j(t))}e^{-\varphi(t, \xi)}  \nonumber\\
+ & \sum_{i, j} a_i(t)\ol{a_j(t)}\int_V \frac{\partial B_t}{\partial \ol w_\gamma}(\xi, s_i(t))\ol{\frac{\partial s_i^\gamma}{\partial t_\alpha}}
  \ol{\frac{\partial B_t}{\partial \ol t_\beta}(\xi, s_j(t))}e^{-\varphi(t, \xi)}  \nonumber\\
+ & \sum_{i, j} a_i(t)\ol{a_j(t)}\int_V \frac{\partial B_t}{\partial \ol w_\gamma}(\xi, s_i(t))\ol{\frac{\partial s_i^\gamma}{\partial t_\alpha}}
  \ol{\frac{\partial B_t}{\partial \ol w_\mu}\ol{\frac{\partial s_j^\mu}{\partial t_\beta}}(\xi, s_j(t))}e^{-\varphi(t, \xi)}. \nonumber\\
\nonumber                                        
\end{align}
\smallskip

\noindent In order to arrange a bit the terms above we make the following observations
\begin{enumerate}
\smallskip

\item[(i)] The second term in the rhs of \eqref{sh19} equals zero, by the formula
  \eqref{sh17}.
\smallskip

\item[(ii)] We introduce the notation
\begin{equation}\label{sh22}\nonumber
\Xi_{i\ol \alpha}(t, \xi):= \ol{\frac{\partial a_i(t)}{\partial t_\alpha}}+ a_i(t)\left(\frac{\partial B_t}{\partial \ol t_\alpha}(\xi, s_i(t))+ \frac{\partial B_t}{\partial \ol w_\mu}(\xi, s_i(t))\ol{\frac{\partial s_i^\mu}{\partial t_\alpha}}\right).
\end{equation}
Then we have
\begin{align}\label{sh23}
\frac{\partial^2B_t\langle a, a\rangle}{\partial t_\beta\partial \ol t_\alpha } = & \sum_{i, j}\int_V\Xi_{i\ol \alpha}(t, \xi)\ol{\Xi_{i\ol \beta}(t, \xi)}e^{\varphi}\\
+ & \sum_{i, j} a_i(t)\ol{a_j(t)}\int_V \frac{\partial^\varphi}{\partial t_\beta}\Big(\frac{\partial B_t}{\partial \ol t_\alpha}\Big)(\xi, s_i(t))
    \ol{B_t(\xi, s_j(t))}e^{-\varphi(t, \xi)} \nonumber\\
+ & \sum_{i, j} a_i(t)\ol{a_j(t)}\int_V \frac{\partial^\varphi}{\partial t_\beta}
      \Big(\frac{\partial B_t}{\partial \ol w_\gamma}\Big)(\xi, s_i(t))\ol{\frac{\partial s_i^\gamma}{\partial t_\alpha}}
    \ol{B_t(\xi, s_j(t))}e^{-\varphi(t, \xi)}  \nonumber\\
  \nonumber \end{align}
and since the following holds
\begin{equation}\label{sh24}\nonumber
\sum_{i, j, \alpha,\beta}\omega^{\ol\alpha \beta}\int_V\Xi_{i\ol \alpha}(t, \xi)\ol{\Xi_{j\ol \beta}(t, \xi)}e^{\varphi}\geq 0
\end{equation}
for each $t\in U$ we only have to deal with the last two terms of the previous equality \eqref{sh23} in order to obtain a lower bound of \eqref{sh18}.
\smallskip

\item[(iii)] The last two terms of \eqref{sh23} are obtained by applying the operator
  $\displaystyle \frac{\partial^\varphi}{\partial t_\beta}\circ \frac{\partial}{\partial \ol t_\alpha}$ to the function $t\to B_t(\xi, s_i(t))$. The commutation formula reads as
\begin{equation}\label{sh25}\nonumber  
\left[\frac{\partial^\varphi}{\partial t_\beta}, \frac{\partial}{\partial \ol t_\alpha}\right]= \frac{\partial^2\varphi}{\partial t_\beta\partial \ol t_\alpha }
\end{equation}
so all in all we infer that we have
\begin{align}\label{sh26}
\sum\omega^{\ol\alpha\beta}\frac{\partial^2B_t\langle a, a\rangle}{\partial t_\beta\partial \ol t_\alpha }\geq & \sum a_i(t)\ol{a_j(t)}\omega^{\ol\alpha\beta}\int_V\frac{\partial^2\varphi}{\partial t_\beta\partial \ol t_\alpha }B_t(\xi, s_j(t))
  \ol{B_t(\xi, s_j(t))}e^{-\varphi(t, \xi)} \nonumber\\
+ & \sum a_i(t)\ol{a_j(t)}\omega^{\ol\alpha\beta}\int_V \frac{\partial }{\partial \ol t_\alpha}\Big(\frac{\partial^\varphi B_t}{\partial t_\beta}\Big)(\xi, s_i(t))
    \ol{B_t(\xi, s_j(t))}e^{-\varphi(t, \xi)} \nonumber\\
\nonumber
\end{align}
\end{enumerate}  

\noindent By the equation \eqref{sh17} we infer that
\begin{equation}\label{sh27}
\int_V \frac{\partial }{\partial \ol t_\alpha}\Big(\frac{\partial^\varphi B_t}{\partial t_\beta}\Big)(\xi, s_i(t))
    \ol{B_t(\xi, s_j(t))}e^{-\varphi(t, \xi)} = -\int_V \frac{\partial^\varphi B_t}{\partial t_\beta}(\xi, s_i(t))
    \ol{\frac{\partial^\varphi B_t}{\partial t_\alpha}(\xi, s_j(t))}e^{-\varphi(t, \xi)} 
\end{equation}
\medskip  

\noindent 
The computations above were done with respect to an arbitrary coordinate system, so we can as well
assume that $\displaystyle \omega^{\ol\alpha \beta}= \delta_{\alpha\beta}$ at some point
$t_0\in U$.
We therefore get 
\begin{align}\label{sh28}
\sum\frac{\partial^2B_t\langle a, a\rangle}{\partial t_\alpha\partial \ol t_\alpha }(t_0)\geq & \sum a_i(t_0)\ol{a_j(t_0)}\int_V\frac{\partial^2\varphi}{\partial t_\alpha\partial \ol t_\alpha }B_{t_0}(\xi, s_i(t_0))
  \ol{B_{t_0}(\xi, s_j(t_0))}e^{-\varphi(t_0, \xi)} \nonumber\\
- & \sum a_i(t_0)\ol{a_j(t_0)}\int_V \frac{\partial^\varphi B_{t_0}}{\partial t_\beta}(\xi, s_i(t_0))
    \ol{\frac{\partial^\varphi B_{t_0}}{\partial t_\beta}(\xi, s_j(t_0))}e^{-\varphi(t_0, \xi)}
\nonumber\\
= & \sum_\alpha\int_V\left(\frac{\partial^2\varphi}{\partial t_\alpha\partial \ol t_\alpha }(t_0, \xi)|\Gamma(\xi)|^2- |\Lambda_\alpha(\xi)|^2\right)e^{-\varphi(t_0, \xi)},\nonumber \\
\nonumber
\end{align}
where the notations used are as follows
\begin{equation}\label{sh29}
\Gamma(\xi):= \sum_i a_i(t_0)B_{t_0}(\xi, s_i(t_0))
\end{equation}
and   
\begin{equation}\label{sh30}
\Lambda_\alpha(\xi):= \sum_{i} a_i(t_0)\frac{\partial ^\varphi B_{t_0}}{\partial t_\alpha}(\xi, s_i(t_0))
\end{equation}
In order to evaluate the norm of $\Lambda_\alpha$ we compute its $\dbar$, and obtain
\begin{equation}\label{sh31}
\dbar\Lambda_\alpha(\xi)= -\sum_{i, k} a_i(t_0)B_{t_0}(\xi, s_i(t_0))\frac{\partial^2\varphi}{\partial t_\alpha\partial \ol\xi_k}d\ol \xi_k
\end{equation}
and we see that this can be rewritten as
\begin{equation}\label{sh32}
\dbar\Lambda_\alpha(\xi)= -\Gamma(\xi)\sum_{k} \frac{\partial^2\varphi}{\partial t_\alpha\partial \ol\xi_k}d\ol \xi_k.
\end{equation}
Moreover, $\Lambda_\alpha$ is perpendicular on the space of holomorphic $L^2$ functions --by the property \eqref{sh17}--, and H\"ormander estimates show that we much have
\begin{equation}\label{sh33}
\int_V|\Lambda_\alpha(\xi)|^2e^{-\varphi(t_0, \xi)}\leq \int_V|\Gamma|^2\frac{\partial^2\varphi}{\partial t_\alpha\partial \ol\xi_k}\frac{\partial^2\varphi}{\partial \ol t_\alpha\partial \xi_m}\varphi^{\ol k m}e^{-\varphi(t_0, \xi)}.
\end{equation}
We thus get the inequality
\begin{equation}\label{sh34}
\sum_\alpha\frac{\partial^2B_t\langle a, a\rangle}{\partial t_\alpha\partial \ol t_\alpha }(t_0)\geq  \sum_\alpha\int_V|\Gamma|^2\left(\frac{\partial^2\varphi}{\partial t_\alpha\partial \ol t_\alpha }- \frac{\partial^2\varphi}{\partial t_\alpha\partial \ol\xi_k}\frac{\partial^2\varphi}{\partial \ol t_\alpha\partial \xi_m}\varphi^{\ol k m}\right)e^{-\varphi(t_0, \xi)}.
\end{equation}
It turns out that the quantity
\begin{equation}\label{sh35}
  \frac{\partial^2\varphi}{\partial t_\alpha\partial \ol t_\alpha }- \frac{\partial^2\varphi}{\partial t_\alpha\partial \ol\xi_k}\frac{\partial^2\varphi}{\partial \ol t_\alpha\partial \xi_m}\varphi^{\ol k m}
\end{equation}
has an intrinsic interpretation: it is equal to
\begin{equation}\label{sh36}
\frac{n}{d+1}\frac{(\ddc \varphi)^{d+1}\wedge p^\star\omega^{n-1}}{(\ddc \varphi)^{d}\wedge p^\star\omega^{n}}
\end{equation}
evaluated at the point $t_0$. By hypothesis, $\varphi$ is psh, which combined with the lower bound for the trace of $\ddc \varphi$ with respect to $p^\star\omega$
shows that
\begin{equation}\label{sh37}
\frac{n}{d+1}\frac{(\ddc \varphi)^{d+1}\wedge p^\star\omega^{n-1}}{(\ddc \varphi)^{d}\wedge p^\star\omega^{n}}\geq n\ep_0.
\end{equation}
The inequality \eqref{sh37} is easy to justify in our situation, since $\varphi$ is
strictly psh, and we can choose the coordinates in such a way that \eqref{sh37} becomes trivial. It holds however true in more general circumstances, as we will now see: it is enough to assume that $\varphi$ to be strictly psh on the fibers of $p$.

To this end, let $(z_i)$ be local coordinates centered at the origin, such that
\begin{align}\label{sh70}
  \ddc \varphi= & \sqrt{-1}\Big(\sum_{\alpha, \beta} \varphi_{\alpha\ol\beta}dt_\alpha\wedge d\ol t_\beta+ \sum_{\alpha, a} \varphi_{\alpha\ol a}dt_\alpha\wedge d\ol z_a+ \sum_{a, \beta} \varphi_{a\ol\beta}dz_a\wedge d\ol t_\alpha\Big)\\
    + &  \sqrt{-1}\sum_{a} dz_a\wedge d\ol z_a.\nonumber \\
\nonumber
\end{align}
By hypothesis we have
\begin{align}\label{sh71}
  n\ep_0\prod_\alpha \sqrt{-1}dt_\alpha d\ol t_\alpha\leq & \sum_{\alpha} \varphi_{\alpha\ol \alpha}\prod_\alpha \sqrt{-1}dt_\alpha d\ol t_\alpha\nonumber \\
  + & \sqrt{-1}\sum_{\alpha, a} \varphi_{\alpha\ol a}dt_\alpha\wedge d\ol z_a
      \prod \wh{dt_\alpha\wedge d\ol t_\alpha}\nonumber \\
+ & \sqrt{-1}\sum_{a, \alpha} \varphi_{a \ol \alpha}dz_a\wedge d\ol t_\alpha
    \prod \wh{dt_\alpha\wedge d\ol t_\alpha}\nonumber \\
+ &  \sqrt{-1}\Big(\sum_{a} dz_a\wedge d\ol z_a\Big)\prod \wh{dt_\alpha\wedge d\ol t_\alpha},\nonumber \\
\nonumber  
\end{align}
and therefore for any choice of complex numbers $\displaystyle (\lambda_{\alpha \ol a})$ we have
\begin{equation}\label{sh72}
\sum_{\alpha} \varphi_{\alpha\ol \alpha}- n\ep_0- 2\Re(\sum_{\alpha, a} \varphi_{\alpha\ol a}\ol{\lambda_{\alpha \ol a}})+ \sum|\lambda_{\alpha \ol a}|^2\geq 0. 
\end{equation}
This now implies that  
\begin{equation}\label{sh73}
\sum_{\alpha} \varphi_{\alpha\ol \alpha}- \sum_{\alpha, a} |\varphi_{\alpha\ol a}|^2\geq n\ep_0
\end{equation}
which is what we wanted to prove.

\noindent The quantity we have on the LHS of \eqref{sh35} reads as
\begin{equation}\label{sh50}
n\frac{\ddc B_t\langle a, a\rangle\wedge \omega^{n-1}}{\omega^n}
\end{equation}
and we get
\begin{equation}\label{sh51}
\ddc B_t\langle a, a\rangle\wedge \omega^{n-1}\geq \ep_0B_t\langle a, a\rangle\omega^n.
\end{equation}
In conclusion, the local version of our result holds true and Theorem \ref{thm0} is proved. 
\medskip

\subsection{Proof of Theorem \ref{thm01}} We first assume that $\omega$ is K\"ahler. Then Theorem \ref{thm01} is easily reduced to Theorem \ref{thm0} as follows. 

\noindent Let $U\subset X$ be an open subset of
$X$, such that
\begin{equation}\label{sh52}
\omega_X |_U= \ddc \varphi_U
\end{equation}
for some function smooth $\varphi_U$ defined on $U$. We then consider the restriction of  $L$ on the $p$-inverse image of $U$ denoted by $Y_U$ and we endow it with the metric
\begin{equation}\label{sh53}
h_1:= e^{-C\cdot \phi}h_L
\end{equation}
where $\phi:= \varphi_U\circ p$. The hypothesis of Theorem \ref{thm0} are fulfilled. The constant $\ep_0$ is in this case equal to $C$. Then we have
\begin{equation}\label{sh54}
\Theta_{\det h_{1, Y_U/U}}(\det\cE)\wedge \omega_X ^{n-1}\geq rC\omega_X ^n,
\end{equation}
which is equivalent to
\begin{equation}\label{sh55}
\Theta_{\det h_{Y/X}}(\det\cE)\wedge \omega_X ^{n-1}\geq 0 ,
\end{equation}
by using the relation $h_{1, Y_U/U} = h_{Y/X} \cdot e^{-C\cdot \varphi_U}$.
\smallskip

\noindent Now for the case of an arbitrary metric $\omega$ we argue as in the proof of Theorem \ref{thm0}. Let $B(x_i, \ep)$ be a covering of $X$ with balls of radius
$\ep$, such that the condition \eqref{reg1} is satisfied. Then on each ball
we consider the metric
\begin{equation}\label{sh353}
h_{i, L}:= e^{-C(1+ \delta_\ep)|t_i|^2}h_L
\end{equation}
where $(t_i)$ are coordinates on the ball $B(x_i, \ep)$. By the K\"ahler case already
discussed, we obtain
\begin{equation}\label{sh354}
\left(\Theta_{\det h_{Y/X}}(\det\cE)+ rC(1+ \delta_\ep)\omega_i\right)\wedge \omega_i ^{n-1}\geq rC\big(1+ \O(\delta_\ep)\big)\omega_i^n
\end{equation}
on each ball $B(x_i, \ep)$. This inequality is basically unchanged if we replace
$\omega_i$ by the global metric $\omega$, again thanks to \eqref{reg1}.
We therefore obtain
\begin{equation}\label{sh355}
\left(\Theta_{\det h_{Y/X}}(\det\cE)+ rC(1+ \delta_\ep)\omega\right)\wedge \omega ^{n-1}\geq rC\big(1+ \O(\delta_\ep)\big)\omega^n
\end{equation}
Theorem \ref{thm01} is therefore established by letting $\ep\to 0$.

\subsection{K\"ahler version of Lemma \ref{trivia}}
We will establish here the K\"ahler version of Lemma \ref{trivia}. The statement is absolutely the same, 
except that we only assume the manifold $X$ to be compact K\"ahler. Also, the arguments are similar to those used in the projective case but we apply Theorem \ref{thm01} instead of Theorem \ref{push}.
There is however an additional slight complication due
to the singularities of $\cF$, so we will provide a complete treatment in what follows.
\smallskip

\noindent Recall that we have the bundle $E\to \wh X$ which admits a sequence of smooth metrics $\displaystyle h_{E, \ep}$ such that
\begin{equation}\label{sh56}
\Theta(E, h_{E, \ep})\wedge g_\ep^{n-1}= \delta_\ep g_\ep^n\otimes \Id_{\End(E)}
\end{equation}
for some constant $\delta_\ep$ such that $\lim_{\ep\to 0}\delta_\ep= 0$.
Let $h$ be a metric on $\O (1)$ which could be singular, but such that
\begin{equation}\label{sh57}
\Theta(\O (1), h)\geq 0
\end{equation}
in the sense of currents on $\P(E)$.
\smallskip

\noindent Assume by contradiction that the multiplier sheaf  $\cI((k+1 -\delta_0)h |_{\P(E^\star)_x})$ is non-trivial on the generic fiber $\P(E)_x$ 
for some $\delta_0 >0$. Then we consider the relative
adjoint bundle
\begin{equation}
K_{\P(E)/\wh X}+ \O (r+k) .
\end{equation}
By a standard $L^2$-estimate (cf. Lemma \ref{l2est}), the direct image
\begin{equation}\label{sh58}
\cG := p_\star\left((K_{\P(E)/\wh X}+ \O (r+k))\otimes \cI((k+1 -\delta_0)h)\right)
\end{equation}
is non zero. We define the metric 
\begin{equation}\label{sh59}
h_\ep:= h_{E, \ep}^{\otimes(r-1+\delta_0)}\otimes h^{\otimes (k+1-\delta_0)} 
\end{equation}
on the bundle $\O(r+k)$ --here we denote with the same symbol $h_{E, \ep}$ the metric on $E$ and the metric induced by it on $\O (1)$.

\noindent Then a direct computation shows that we have
\begin{equation}\label{sh60}
  \Theta(\O (r+k), h_\ep)\wedge \pi^\star g_\ep^{n-1}\geq (r-1) \delta_\ep
  \pi^\star g_\ep^{n}.
\end{equation}
Here we are using \eqref{sh57} combined with the Hermite-Einstein identity
\eqref{sh56} and the explicit expression of the curvature of $\O (1)$.
\smallskip

\noindent We therefore get a \emph{proper} subsheaf $\cG$ of $\Sym^k E^\star \otimes \det E^\star$ for which we have
\begin{equation}\label{sh61}
\int_{\wh X}c_1(\cG )\wedge (\pi^\star g_\ep)^{n-1}\geq
r (r-1)\delta_\ep\int_{\wh X}(\pi^\star g_\ep)^{n}  
\end{equation}
by Theorem \ref{thm01}. Moreover, $\cG$ is independent of $\ep$. As $\ep\to 0$ we obtain
\begin{equation}\label{sh62}
\int_{\wh X}c_1(\cG)\wedge \pi^\star \omega^{n-1}\geq
0
\end{equation}
since the volume of $\wh X$ with respect to $g_\ep$ is uniformly bounded.
\smallskip

\noindent As in the projective case, the stability hypothesis forbids the existence
of such subsheaf $\cG$. In conclusion, we have
\begin{equation}\label{sh63}
 \cI((k+1-\ep) h |_{\P(E)_x} )= \O_{\P(E)_x} ,
\end{equation}
for any $\ep >0$, where $x$ is a generic point.

\bigskip

\noindent Finally, we prove Theorem \ref{thm1} in the case when $X$ is compact K\"{a}hler. Thanks to the stability condition $(ii)$ of Theorem \ref{thm1}, \eqref{sh63} implies that
$$\cI(h^{\otimes (r+1)}  |_{\P(E)_x} )= \O_{\P(E)_x} .$$
By applying Theorem \ref{push} to 
$$ p_\star \left( (K_{\P (E) /X} + \mathcal{O}(r+1)) \otimes \cI(h^{\otimes (r+1)}) \right) = \cF^\star ,$$
$\cF^\star$ is positively curved. Together with $c_1 (\cF^\star) =0$, Proposition \ref{regsh} implies that $\cF^\star$ is hermitian flat. 
\medskip

\noindent In our previous arguments we have used the following statement, which is
a consequence of $L^2$-estimates. 
\begin{lemme}\label{l2est}
Let $\mathbb{P}^n$ be the $n$-dimensional projective space, and let $m\in\mathbb{N}^\star$. Let $h$ be a singular metric on $\mathcal{O}_{\mathbb{P}^n}(m)$ such that $i\Theta_{h} (\mathcal{O}_{\mathbb{P}^n} (m)) \geq 0$ on $X$.
Then the space $H^0 (\mathbb{P}^n , \mathcal{O}_{\mathbb{P}^n} (m-1)\otimes \mathcal{I} ((1-\ep)h))$ is non zero for every $\ep >0$.
\end{lemme}

\begin{proof}
Set $L:= \mathcal{O}_{\mathbb{P}^n} (m+n)$. Then we have
$$\mathcal{O}_{\mathbb{P}^n} (m-1) = K_{\mathbb{P}^n} +L .$$
Let $x\in \mathbb{P}^n \setminus V (\mathcal{I} ((1-\ep)h))$ and let $\{s_1, s_2,
\dots, s_n\}$ be a basis of $H^0 (\mathbb{P}^n, \O (1))$ vanishing on $x$.
The local weight 
$$\varphi := n \log (\sum_i |s_i|^2)$$
defines a metric on $\O (n)$ of isolated singularity at $x$ with Lelong number $n$.
We define the metric
$$h_L := h_\varphi + (1-\ep)h + \ep h_{FS}$$
on $L$, where $h_{FS}$ is the Fubini-Study metric. Then the $L^2$-estimates (cf. \cite[Cor 5.12]{Dem}) implies that
$$H^0 (\mathbb{P}^n , \mathcal{O}_{\mathbb{P}^n} (K_{\mathbb{P}^n} +L)\otimes \mathcal{I} ((1-\ep)h)) \rightarrow (\mathcal{O}_{\mathbb{P}^n} (K_{\mathbb{P}^n} +L))_x$$
is surjective. The lemma is proved.
\end{proof}

\section{A few comments about the general conjecture and other results}

\noindent We recall that the conjecture of Pereira-Touzet cf. \cite{PeTou} states that if $\cF\subset T_X$ is a holomorphic foliation such that $c_1(\cF)= 0$, $c_2(\cF)\neq 0$ and $\cF$ is $\omega_X$-stable and $K_X$ is pseudo-effective, then $\cF$ is algebraic. 
\smallskip
    
\noindent As we have already mentioned, under the conditions above $\cF$ is a sub-bundle of $T_X$. If we assume (by contradiction) that $\cF$ is not algebraic, then
the bundle $\O(1)\to \P(\cF)$ is pseudo-effective. We can therefore
find a singular metric $e^{-\varphi}$ on it such that
the curvature $\Theta =\ddc\varphi \geq 0$ in the sense currents. The Siu decomposition of this current writes as
\begin{equation}\label{eqn1}
  \Theta = \sum a_i [E_i] + T,
\end{equation}
where $ \sum a_i [E_i] $ is the divisorial part and $T$ is a positive current with the property that $\codim_{X} E_c (T) \geq 0$ for every $c>0$. Here $E_c (T)$ is the locus where the Lelong number of $T$ is at least $c$.

\noindent We show now the following result concerning the divisorial part of the current $\Theta$.

\begin{thm}\label{h_div}
Let $X$ be a smooth projective manifold, such that $K_X$ is pseudo-effective. We assume that $\cF$ is a holomorphic foliation on $X$ such that $c_1(F)= 0$. Assume moreover that $\cF$ is $H$--stable, where $H$ is an ample divisor
on $X$. If $\O(1)$ is pseudo-effective and let $E$ be one component of the divisorial part of \eqref{eqn1},  then the map
\begin{equation}\label{eqn-1}
p_E:E\to X
\end{equation}
induced by the restriction of $p$ to $E$ is a locally trivial fibration. 
\end{thm}
\begin{proof}
\noindent We first remark that for each $j$ we have
\begin{equation}\label{eqn2}
E_j\equiv m_j\O(1)+ p^\star(L_j)
\end{equation}
where $m_j$ is a positive integer, $L_j$ is a line bundle on $X$ and $p:\P(T_\cF)\to X$ is the projection map. We equally have
\begin{equation}\label{eqn3}
T\equiv m_T\O(1)+ p^\star(L_T)
\end{equation}
where $m_T\geq 0$ is a real number and $L_T$ is a $\R$-bundle.

The hypothesis $c_1(\cF)= 0$ together with the stability condition
shows that for each $j$ we have 
\begin{equation}\label{eqn4}
L_j\cdot H^{n-1}\geq 0
\end{equation}  
as soon as $m_j\geq 1$, i.e. if $E_j$ is horizontal with respect to $p$. Now if
$E_j$ projects into a proper subvariety of $X$, then we also have
$\displaystyle L_j\cdot H^{n-1}\geq 0$ since it means that the corresponding $m_j$ equals zero. Moreover, $T$ is a closed positive current, so a quick approximation argument shows that $\displaystyle L_T\cdot H^{n-1}\geq 0$ as well.
On the other hand, the equality \eqref{eqn1} shows that numerically we have
\begin{equation}\label{eqn5}
\sum \nu_j L_j+ L_T\equiv 0 .
\end{equation}
Hence we infer that $L_j \cdot H^{n-1} = L_T\cdot H^{n-1} =0$ for every $j$.
Therefore, we obtain that $m_j \geq 1$ for every $j$, namely, every component $E_j$ is horizontal with respect to $p$. 

\medskip

\noindent Consider the exact sequence
\begin{equation}\label{eq99}
0\to L_j^{-1}\to \Sym^{m_j}\cF^\star\to Q_j\to 0
\end{equation}
induced by the section $E_j$. By taking the determinants and using the previous considerations we obtain $\det(Q_j)\wedge \omega^{n-1}= 0$. 
On the other hand, the foliation $\cF$ is smooth so it follows that
$\det Q_j$ is pseudo-effective, cf. \cite{CP19}. In conclusion, $c_1(Q_j)= 0$,
which in turn implies that $c_1(L_j)= 0$.
\smallskip

%\noindent \emph{Conclusion}: The bundle $\Sym^m\cF^\star$ admits a holomorphic section (up to the twist with a topologically trivial line bundle).
%Therefore, we can assume  that
%the metric $e^{-\varphi}$ corresponds to an irreducible $\mathbb Q$-divisor. This is so because if we only have vertical components we are done. Also, if we have two horizontal components, then we consider
%\begin{equation}\label{eq007}
%\max\left\{\frac{1}{m_1}\varphi_{E_1}, \frac{1}{m_2}\varphi_{E_2}\right\}
%\end{equation}
%and obtain in this way a positively curved metric on $\O(1)$ whose singularities are in codimension two.
%\smallskip

%\noindent We show next that the only case to consider is $m= 2$. Indeed, let $k$ be a positive integer. The direct image
%\begin{equation}\label{comm1}
%p_\star\left((kK_{\P(\cF^\star)/X}+ \O(2k+1))\otimes \cI(h^{\frac{1+ 2k}{km}})\right)
%\end{equation}
%i.e. sections which are in $L^{2/k}$ with respect to the metric in \eqref{comm1}
%is also semi-positively curved in the sense of Griffiths. Since the intersection of the divisor $E$ with the generic fiber is smooth, it would be enough to have 
%\begin{equation}\label{comm2}
%\frac{1+ 2k}{km}< 1
%\end{equation}
%which is certainly the case if $m\geq 3$ and $k=2$. The remaining case is $m=2$.  
%\medskip

\noindent  
After these preliminary considerations, we are ready to prove that $p_E$ is locally trivial. Let
$$S:=H_1 \cap H_2 \cap \cdots \cap H_{n-1}$$ 
be the complete intersection of some smooth hypersurfaces $H_i \in |k H|$ for $k$ large enough. To prove that $p_E$ is locally trivial, it is sufficient  to prove that 
$$p_E |_{p_E ^{-1}(S)}: p^{-1}(S) \cap E \rightarrow S$$ 
is locally trivial.  

Thanks to \cite[Thm 5.2]{Lan04}, $\cF |_S$ is stable if $k$ is large enough (note that we don't need ask that $H_i$ is a generic hypersurface here). Since $c_1 (\cF |_S)=0$, it follows that $\O _{\P (\cF)}(1) |_S$ admits a smooth metric $h$ of semi-positive curvature.  As we proved that  $L_j \cdot H^{n-1}=0$ for every $j$,  we get
$$c_1 (E) = c_1 (\O (m)) \in H^{1,1} (\P (\cF), \R) $$
for some $m\in \mathbb{N}$.
Let $a\geq m$ be some number large enough.  $\O (a)  |_S$ be can equipped with the metric
$$h_{a} := h^{a-m} \cdot e^{-\log |s_{E}|^2}, $$
whose curvature is semipositive. 
We consider the direct image of $p: p^{-1}(S) \rightarrow S$ 
$$p_\star ( (K_{p^{-1}(S)/S} + \O_{\P (\cF)}(a) )\otimes \mathcal{I} (h_a)) \subset p_\star (K_{p^{-1}(S)/S} + \O _{\P (\cF)}(a) ) \qquad\text{on } S .$$
Then RHS is just $\Sym^{a-2} \cF^\star |_S$, which is hermitian flat. The LHS is Griffiths-semi-positive by \cite{PT}.
Therefore LHS is also hermitian flat, and the flat connection is compatible with the flat connection on $\cF^\star |_S$.
Note also that the germs of $p_\star (K_{p^{-1}(S)/S} + \O _{\P (\cF)}(a) \otimes \mathcal{I} (h_a) )_s$ is just the degree $(a-r)$-polynomials vanishing on $ E$.
For $a$ large enough, the flatness $p_\star (K_{p^{-1}(S)/S} + \O _{\P (\cF)}(a) \otimes \mathcal{I} (h_a) )$ implies that the common zero locus of these polynomials (which is $E$) is also 
invariant by the flat connection on $\cF |_S$. It means that $E |_S$ is invariant by the parallel tranport of the flat connection on $\cF |_S$. 
Then $p_E ^{-1}(S)  \rightarrow S$ is locally trivial.  As $S$ is the complete intersection of any smooth hypersurfaces in $|k H|$, $p_E$ is thus locally trivial.
\end{proof}  
\medskip

\noindent If the rank of $\cF$ is equal to two we obtain the following result, which is already known cf. \cite{Dr17}, but our arguments here are different.

\begin{cor} Let $X$ be a projective manifold such that $K_X$ is pseudo-effective and let $\cF\subset T_X$ be a holomorphic subsheaf of rank $2$ such that the following hold.
\begin{enumerate}

\item[(i)] The first Chern class of $\cF$ is zero, i.e. $c_1(\cF)= 0$.
  \smallskip

\item[(ii)] The sheaf $\cF^\star$ is $\omega_X$-strongly stable, meaning that
for any finite \'{e}tale cover $\pi: X' \rightarrow X$, $\pi^\star\cF^\star$ is $\pi^\star\omega_X$-stable.
\end{enumerate}  
Then the bundle $\O(1)$ on $\P(\cF)$ is not pseudo-effective, or $\cF$ is Hermitian flat.
\end{cor}

\begin{proof}
  \noindent If the bundle $\O(1)$ on $\P(\cF)$ is pseudo-effective, we consider Siu's decomposition \eqref{eqn1}. If there is no divisorial part, namely $\sum a_i [E_i] =0$, then the Lelong number of $T$ vanishes over the generic fiber on the projection $p$ (since the fibers are of dimension $1$). We have already explained that in this case it follows that $\cF$ is hermitian flat.

  \noindent If we assume moreover that $\cF$ is not hermitian flat, we have to have a component $E$ of the divisor-like part of the current $\Theta$. Thanks to Theorem \ref{h_div}, the map
$$p_E: E\to X$$ is an \'etale cover.
  We prove that  $p_E ^\star \cF^\star$ is not stable and this contradicts with our strongly stable condition.
  Indeed, we have the arrows
\begin{equation}\label{eqn33}
0\to \O(-1)|_E\to p_E^\star(\cF^\star)
\end{equation}    
and a quick computation shows that we have
\begin{equation}\label{eqn34}
\int_Ec_1(\O(-1))\wedge p_E^\star(\omega_H^{n-1})= 0,
\end{equation}
because $c_1(\cF)= 0$.  
This is the end of the proof.
\end{proof}

\medskip

\begin{remark}{\rm Actually Claim 4.1 (on page 7) concerns the
    positivity of the curvature form of $\cE$ on the current $\omega_X^{n-1}$. It
    would be nice to extend this as initiated by Berndtsson-Sibony in \cite{BeSib}
    i.e. replace $\omega_X^{n-1}$ with more general currents of bidimension $(1,1)$ on $X$. The following two results are pointing in this direction.
  }
\end{remark}    

\medskip

\begin{thm}\label{bermganprop}
Let $p:Y\to X$ be a holomorphic surjective and proper map, where $X$ and $Y$ are K\"ahler manifolds. 
Moreover, we assume that $p$ is locally projective. Let $L$ be a line bundle over $Y$ with a possible singular metric $h_L$ such that  $i\Theta_{h_L} (L)\geq 0$, and 
\begin{equation}\label{cor030}  
\Theta_{h_L}(L)\wedge p^\star\omega_X^{n-1}\geq \ep_0p^\star\omega_X^{n}
\end{equation}
where $\omega_X$ is a Hermitian metric on $X$, $n=\dim X$, and $\ep_0> 0$ is a positive real number. We assume that the space of fiberwise $\displaystyle L^{\frac{2}{m}}$ sections (with respect to $h_L$) of
$p_\star (m K_{Y/X} +L)$ is non zero. Then there exists a metric $h$ on the
bundle $m K_{Y/X} +L$ such that we have
\begin{equation}\label{fiberversion}
\Theta_{h} (m K_{Y/X} +L) \wedge p^\star \omega_X ^{n-1} \geq \ep_0 \cdot p^\star \omega_X ^n 
\end{equation}
in the sense of currents on $X$.
\end{thm}

\begin{proof}
  We first show that this is true for $m=1$. In this case the metric $h$ in our statement \ref{bermganprop} is precisely the fiberwise Bergman metric, as we will now see.
\smallskip
  
\noindent Actually, this a consequence of the proof of Theorem \ref{thm01}. We follow the same approximation process for the map $p$ restricted to one of the balls
$B(x_i, \ep)$, on which the distortion between $\omega_X$ and the flat metric
in coordinates $(t_i)$ is $1+ \O(\delta_\ep)$. The inequality \eqref{reg14} applied for the sections $s_1 =s_2=\cdots =s_r =s$ shows
that the regularized version of \eqref{fiberversion} holds true. Indeed \eqref{reg14} holds true when restricted to the image of any section $s$, hence point-wise
by \cite[III Criteria 1.6]{Dem}.

\noindent In conclusion, the inequality 
\begin{equation}\label{nncurentpositive}
  \Theta_{h_B} (K_{Y/X} +L) \wedge p^\star \omega_X ^{n-1}\geq  \ep_0 p^\star \omega_X ^n
\end{equation}  
follows by taking the limit of several parameters involved.
  
%Thanks to \cite{BP08}, $i\Theta_{h_B} (K_{Y/X} +L)$ is psh. It would be thus enough to prove that \eqref{fiberversion} holds true when restricted to the smooth locus of
%the fiberation. In particular, we can assume that $p$ is smooth. 

%Let $U\subset X$ be some open set and $s$ be section of the fibration $p: p^{-1} (U) \rightarrow U$. 
%By taking , \eqref{equberg} implies that 

%To prove \eqref{fiberversion}, thanks to \cite[III Criteria 1.6]{Dem}, it is enough to prove that the restrction of 
%$i\Theta_{h_B} (K_{Y/X} +L) \wedge p^\star \omega_X ^{n-1} - \ep_0 p^\star \omega_X ^n$ on any small $n$-dimensional disc $\Delta$ in $Y$ is positive.
%There are two cases.

%\medskip

%\noindent If $\Delta$ is horizental, namly $p_\star  (\Delta)$ is $n$-dimensional, $\Delta$ can seen as the zero locus of a local section of the fibration $p$.
%Then \eqref{nncurentpositive} implies that 
%$$(i\Theta_{h_B} (K_{Y/X} +L) \wedge p^\star \omega_X ^{n-1} - \ep_0 p^\star\omega_X ^n) |_\Delta \geq 0 .$$
%If $\Delta$ is not horizental, we have $ p^\star\omega_X ^n  |_\Delta \equiv 0$. By construction, the restrction of $i\Theta_{h_B}$ on the generic fiber is positive.
%Then 
%$$(i\Theta_{h_B} (K_{Y/X} +L) \wedge p^\star \omega_X ^{n-1}) |_\Delta \geq 0.$$
%Therefore we have 
%$$(i\Theta_{h_B} (K_{Y/X} +L) \wedge p^\star \omega_X ^{n-1} - \ep_0 p^\star\omega_X ^n) |_\Delta \geq 0 .$$
%\eqref{fiberversion} is thus proved.

\bigskip

\noindent If $m\geq 2$ we argue as follows. Let $h_B$ be the relative $m$-Bergman kernel metric on $m K_{Y/X} +L$. By our hypothesis, this is not identically $+\infty$. We define the bundle $F := (m-1) K_{Y/X} +L$. Then the local weights
\begin{equation}\label{rem1}
  \varphi_1 := \big(1-\frac{1}{m}) \varphi_B + \frac{1}{m} \varphi_L
\end{equation}
define a metric on $F$ satisfying
\begin{equation}\label{rem100}\Theta_{h_1} (F) \geq 0,\qquad \Theta_{h_1} (F) \wedge p^\star\omega_X^{n-1}\geq \frac{\ep_0}{m}p^\star\omega_X^{n} .
\end{equation}  
By applying the above $m=1$ case to $(K_{Y/X} +F, h_1)$, we obtain a Bergman type metric $h_{B,1}$ on $K_{Y/X} +F$ such that 
\begin{equation}\label{rem101}\Theta_{h_{B,1}} (K_{Y/X} +F) \geq 0, \qquad \Theta_{h_{B,1}} (K_{Y/X} +F) \wedge p^\star\omega_X^{n-1}\geq \frac{\ep_0}{m}p^\star\omega_X^{n}  .
\end{equation}  
\smallskip

\noindent We now iterate this process, namely we define the metric
\begin{equation}\label{rem2}
  \varphi_2 := \big(1-\frac{1}{m}) \varphi_{B,1} + \frac{1}{m} \varphi_L
\end{equation}
on $F$, and let $h_{B,2}$ be the relative Bergman kernel metric corresponding
to the data $(K_{Y/X} +F, h_2)$. Note that we have
\begin{equation}\label{rem3}
  \Theta_{h_2} (F) \wedge p^\star\omega_X^{n-1}\geq (\frac{1}{m}+\frac{1}{m} (1- \frac{1}{m})) \ep_0 p^\star\omega_X^{n},\end{equation}
we get thus $\displaystyle \Theta_{h_{B,2}} (K_{Y/X} +F) \geq 0$ and
\begin{equation}\label{rem4}
\Theta_{h_{B,2}} (K_{Y/X} +F) \wedge p^\star\omega_X^{n-1}\geq \Big(\frac{1}{m}+\frac{1}{m} (1- \frac{1}{m})\Big) \ep_0 p^\star\omega_X^{n}.
\end{equation}  
Then for every $k\in \mathbb{N}$, we obtain a metric $h_{B, k}$ on $K_{Y/X} +F$ such that $\displaystyle \Theta_{h_{B,k}} (K_{Y/X} +F) \geq 0$ together with
\begin{equation}\label{rem5}
\Theta_{h_{B,k}} (K_{Y/X} +F) \wedge p^\star\omega_X^{n-1}\geq \frac{1}{m} (\sum_{i=0}^{k-1} (1- \frac{1}{m})^i) \ep_0 p^\star\omega_X^{n}. \end{equation}
By normalization (and eventually taking a subsequence), $h_{B,k}$ converges to a metric denoted by $h$ on $K_{Y/X} +F = m K_{Y/X} + L$ which satisfies
\eqref{fiberversion}.   \end{proof}

\bigskip

\noindent We remark that although the limit metric $h$ does not exactly the $m$-Bergman kernel metric, but nevertheless it has the following important property.
Let $D$ be some divisor such that $p^\star (D) =\sum a_i E_i$. By the construction, for every $k\in \N$, we have  
\begin{equation}\label{rem6}
  \Theta_{h_{B,k}} (m K_{Y/X} +L) \geq \sum_i m (a_i -1)^+ [E_i]
\end{equation}  
Then the limit metric $h$ satisfies a similar inequality
\begin{equation}\label{rem7}\Theta_{h} (m K_{Y/X} +L) \geq \sum_i m (a_i -1)^+ [E_i]
  \end{equation}
as well as
\begin{equation}\label{nonreducontrol}
 (\Theta_{h} (m K_{Y/X} +L) - \sum_i m (a_i -1)^+ [E_i]) \wedge p^\star \omega_X ^{n-1} \geq \ep_0 \cdot p^\star \omega_X ^n .
 \end{equation}
\medskip

\noindent We also have the following version of Theorem \ref{bermganprop}.

\begin{cor}\label{coroberg}
Let $p:Y\to X$ be a holomorphic surjective and proper map, where $X$ and $Y$ are K\"ahler manifolds. 
Moreover, we assume that $p$ is locally projective. Let $L$ be a line bundle over $Y$ with a possible singular metric $h_L$ such that we have
\begin{equation}\label{cor130}  
\Theta_{h_L} (L)\geq -Cp^\star\omega_X,\qquad \Theta_{h_L}(L)\wedge p^\star\omega_X^{n-1}\geq 0
\end{equation}
where $\omega_X$ is a Hermitian metric on $X$, $n=\dim X$, and $C> 0$ is a positive real number. We assume that the space of fiberwise $\displaystyle L^{\frac{2}{m}}$ sections (with respect to $h_L$) of
$p_\star (m K_{Y/X} +L)$ is non zero. Then there exists a metric $h$ on the
bundle $m K_{Y/X} +L$ such that we have
\begin{equation}\label{fibversion}
\Theta_{h} (m K_{Y/X} +L) \wedge p^\star \omega_X ^{n-1} \geq 0
\end{equation}
in the sense of currents on $X$.  
\end{cor}
\noindent The proof of this statement follows in the same way we have obtained
Theorem \ref{thm01} as consequence of Theorem \ref{thm0}, so we provide no further explanations about it.
\smallskip

\noindent
We note that as a direct consequence of the arguments we use in the proof of
\ref{bermganprop}, we obtain the existence of a sequence of $m$-Bergman metrics on
$m K_{Y/X} +L$ such that
\begin{equation}\label{rem8}
\Theta_{h_{B,k}} (K_{Y/X} +F) \wedge p^\star\omega_X^{n-1}\geq -\delta_kp^\star\omega_X^{n}
\end{equation}
where the sequence $(\delta_k)$ converges to zero.
\medskip

\noindent Thanks to Corollary \ref{coroberg}, we can prove the following variant of \cite[Thm 3.4]{CP17} (which is using the fundamental contributions of Viehweg, Tsuji... \cite{V95, T10} ). 
\begin{cor}
Let $p:Y\to X$ be a holomorphic surjective map between two K\"ahler manifolds. 
We assume that $p$ is locally projective. Let $\Sigma \subset X$ be the singular locus of $p$ and we assume that both $\Sigma$ and $p^{-1} (\Sigma)$ are normal crossing. 
Let $(L, h_L)$ be a hermitian line bundle over $X$ such that
\begin{equation}\label{corr03}  
\Theta_{h_L}(L)\geq -C p^\star\omega_X, \qquad \Theta_{h_L}(L)\wedge p^\star\omega_X^{n-1}\geq 0
\end{equation}
where $\omega_X$ is a K\"ahler metric on $X$, $n=\dim X$, and $C> 0$ is a positive real number.
If $\cI (h_L |_{Y_x}) = \O_{Y_x}$ for a generic $x\in X$ and $\cE := p_\star (K_{Y/X} +L)$ is non zero, then there exists a divisor $F$ in $Y$ satisfying $\codim_X p_\star (F) \geq 2$, such that
\begin{equation}\label{sh133}
\int_Y c_1\left(K_{Y/X}+ L +F - \delta_0p^\star(\det \cE)\right)\wedge p^\star\omega_X ^{n-1}\geq 0.
\end{equation}
for some positive $\delta_0$.   
\end{cor}

\begin{proof}
The proof is a linear combination of \cite[Thm 3.4]{CP17} with the arguments above. We sketch the proof here for the convenience of readers.

By hypothesis, $p^\star \Sigma$ is normal crossing, and it can be written as
$$p^\star \Sigma = \sum W_i + \sum a_i V_i ,$$
where $ \sum W_i  + \sum V_i $ is snc and $a_i \geq 2$. 
Let $r$ be the rank of $\cE$. Let $Y^r$ be the fiberwise product of $p$, and let $\pr_i: Y^r \rightarrow Y$ be the $i$-directional projection.
Let $\varphi: Y^{(r)} \rightarrow Y^r$ be a desingularisation, and $p^{(r)} : Y^{(r)} \rightarrow X$ be the natural morphism. 
We set $L^{(r)}  := \varphi^\star (\sum_i \pr_i ^\star L)$. 
\smallskip

\noindent We have 
the canonical morphism
$\det \cE \rightarrow \otimes^r \cE$ over the locally free locus of $p$. This induces a section 
$$s\in H^0 (Y^{(r)},   K_{Y^{(r)} /Y} +L^{(r)} - (p^{(r)})^\star\det \cE + E_1 +E_2) ,$$
where $E_1$ and $E_2$ are effective divisors such that $\codim_Y  p^{(r)} _\star (E_2)\geq 2$ and 
$$E_1\leq C \sum_{i,k} (\varphi\circ\pr_i)^\star V_k $$
for some constant $C$. 
Let $m\in\mathbb{N}$ sufficiently large. We consider the line bundle
$$F:=  m L^{(r)} +  (K_{Y^{(r)} /X} +L^{(r)} -(p^{(r)})^\star\det \cE + E_1 +E_2) $$
with the metric $h_F := e^{-\log |s|}h_{L^{(r)}}^{\otimes m}.$
Let $h_k$ be the $m$-relative Bergman kernel metric on $m K_{Y^{(r)} /F} + F$
constructed in \eqref{rem8} with respect to $h_F$.
They satisfy the inequality
\begin{equation}\label{applcorr}
  \Theta_{h_k} (m K_{Y^{(r)} /X} + F)\wedge (p^{(r)})^\star \omega_X ^{n-1} \geq
-\delta_k(p^{(r)})^\star \omega_X ^{n} 
\end{equation}
on $Y^{(r)}$.

Let $\pi: Y \rightarrow Y^{(r)}$ be the fiberwise diagonal embedding.
Note that if $\cI (h_L) =\mathcal{O}_X$ and if $m$ large enough
with respect to $\Div (s)$, we have
$$\mathcal{I} \big(h_F^{\frac{1}{m}} |_{Y^{(r)} _x}\big) = \O_{Y^{(r)} _x}$$
for a generic $x\in X$.
As a consequence, $\pi^\star h_k$ \emph{is not} identically $+\infty$. Therefore 
$\pi^\star \Theta_{h_k} (m K_{Y^{(r)} /X} + F)$ is well defined (thus quasi-psh) on $Y$. 
We have
$$\pi^\star (\Theta_{h_k}) \wedge p^\star \omega_X^{n-1} = \pi^\star \left(\Theta_{h_k} \wedge (p^r)^\star  \omega_X^{n-1}\right) \geq 0 \qquad\text{on } Y,$$
as a consequence of \eqref{applcorr}.
Note that $\pi^\star (\Theta_{h_k})$ belongs to the class of 
$$\pi^\star (c_1 (m K_{Y^{(r)} /X} + F)) = c_1 ((mr+r) ( K_{Y/X} +L) -\pi^\star \det \cE +E_3 +F) ,$$
where $E_3$ is supported $\sum_i V_i$ and $\codim_X p_\star (F) \geq 2$.
Therefore we have
$$\int_Y c_1 ((mr+r) ( K_{Y/X} +L) -\pi^\star \det \cE +E_3 +F) \wedge p^\star
\omega_X^{n-1} \geq -\delta_k p^\star
\omega_X^{n}.$$
Finally, thanks to \eqref{nonreducontrol}, we know that 
$$\sum_i \int_Y  c_1 (V_i) \wedge p^\star \omega_X ^{n-1}  \leq  \int_Y c_1 (K_{Y/X} +L) \wedge p^\star \omega_X ^{n-1}$$
Therefore \eqref{sh133} is proved as $k\to \infty$.
\end{proof}

\end{document}